\documentclass[12pt,a4wide]{amsart}
\usepackage{amsmath,amsfonts,euscript,amscd,amsthm,amssymb,upref,graphics}

\usepackage[T1]{fontenc}
\usepackage{graphics}
\usepackage{amssymb}
\usepackage{enumerate}
\usepackage{appendix}

\def\Gl{\rm Gl}

\def\spa{{\rm span}}

\newcommand{\set}[1]{\left\{{#1}\right\}}

\newcommand{\Aut}{\operatorname{Aut}}
\newcommand{\R}{\mathbb{R}}
\newcommand{\Z}{\mathbb{Z}}
\newcommand{\C}{\mathbb{C}}

\newcommand{\N}{\mathbb{N}}
\newcommand{\B}{\mathbb{B}}

\newcommand{\hol}{\mathcal{O}}

\newcommand{\paren}[1]{\left( {#1}\right)}
\newcommand{\hut}{\Check}
\newcommand{\Sing}{\operatorname{Sing}}
\newcommand{\aut}{\operatorname{Aut}}

\renewcommand{\div}{\operatorname{div}}

\newtheorem{definition}{Definition}[section]
\newtheorem{theorem}{Theorem}[section]
\newtheorem*{theorem*}{Theorem}
\newtheorem{corollary}{Corollary}[section]
\newtheorem{proposition}{Proposition}[section]

\newtheorem{lemma}{Lemma}[section]
\newtheorem*{lemma*}{Lemma}

\theoremstyle{remark}
\newtheorem{notation}{Remark}
\newtheorem{example}{Example}
\begin{document}

\title[ Families of  embeddings and of group actions 
]{Holomorphic families of non-equivalent embeddings and of holomorphic group actions on affine space}

\author{FRANK KUTZSCHEBAUCH}

\address{Institute of Mathematics, University of Bern \\ Sidlerstrasse 5, 
CH-3012 Bern, Switzerland} \email{frank.kutzschebauch@math.unibe.ch}

\author{SAM LODIN}

\address{Dept. of Natural Sciences, Engineering and Mathematics, Mid Sweden University 
\\ SE-851 70 Sundsvall, Sweden} \email{sam.lodin@miun.se}

\thanks{The first author
supported by Schweizerische Nationalfonds grant 200021-116165/1}

\begin{abstract}
We construct holomorphic families of proper holomorphic embeddings of $\C^k$ into $\C^n$
($0<k<n-1$), so that for any two different parameters in the family no holomorphic automorphism of $\C^n$ can map the image of the corresponding two embeddings onto each other. As an application to the study of the group of holomorphic automorphisms of $\C^n$ we derive the existence of families of holomorphic $\C^*$-actions on $\C^n$ ($n\ge 5$)
so that different actions in the family are not conjugate. This result is surprising in view of the long standing
Holomorphic Linearization Problem, which in particular asked whether there would be more than one conjugacy class of $\C^*$  actions on $\C^n$ (with prescribed linear part at a fixed point).
\end{abstract}



\maketitle

\section{Introduction and statement of the main results.}

It is a famous theorem of Remmert that any Stein manifold of dimension 
$n$ admits a proper holomorphic embedding into affine $N$-space $\C^N$ of 
sufficiently high dimension $N$~\cite{R}. Concerning this dimension,
Eliashberg, Gromov~\cite{EG} and Sch\"urmann~\cite{S} proved that any Stein 
manifold of dimension $n>1$ can be embedded into $\C^{[3n/2]+1}$. 
A key ingredient in these results is the homotopy principle for holomorphic 
sections of elliptic submersions over Stein manifolds \cite{Gro}, \cite{FP2}, \cite{FF2}. 
These dimensions are the smallest possible due to an example of 
Forster~\cite{Fs1}. The optimal dimension for  embeddings of Stein spaces can 
be found in Sch\"urmann's paper~\cite{S}.

In this paper we do not investigate the question whether a given Stein space 
can be embedded into $\C^N$ for a given dimension $N$, but rather we 
investigate in how many ways this can be done in situations where at least one 
embedding exists. More precisely, we study the number of equivalence classes of 
proper holomorphic embeddings $\Phi\colon X\hookrightarrow\C^n$ with respect to the 
following equivalence relation:

\begin{definition}\label{def-eq-emb}
Two embeddings $\Phi,\Psi\colon X\hookrightarrow\C^n$ are {\it equivalent} if there 
exist automorphisms $\varphi\in\Aut(\C^n)$ and $\psi\in\Aut(X)$ such that 
$\varphi\circ\Phi=\Psi\circ\psi$.
\end{definition}

In the algebraic case the question about the number of classes of equivalent 
embeddings $\C^k\hookrightarrow\C^n$ is well known and has been studied for a long 
time. The most famous result, due to Abhyankar and Moh~\cite{AM}, states that 
every polynomial embedding of $\C$ into $\C^2$ is equivalent to the standard 
embedding. The same is in general true for high codimension, Kaliman~\cite{Ka} 
proved that if $X$ is an affine algebraic variety and 
$n\geq\max\set{1+2\dim X,\dim TX}$ then all polynomial embeddings of $X$ into 
$\C^n$ are equivalent (by means of algebraic automorphisms). In the same paper 
Kaliman also proved that any polynomial embedding of $\C$ into $\C^3$ is 
holomorphically equivalent to the standard embedding. It is still an open 
question if this holds algebraically.

In the holomorphic case the situation is different. Rosay and Rudin \cite{RR2} 
were the first to construct non-standard embeddings of $\C$ into $\C^n$, $n>2$, 
thus showing that the number of equivalence classes is at least two. 
Forstneri\v c, Globevnik and Rosay~\cite{FGR} showed that the result of Rosay 
and Rudin also holds for $n=2$. More generally, Forstneri\v c~\cite{F} showed 
that the number of equivalence classes of embeddings $\C^k$ into $\C^n$ is at 
least two for any $0<k<n$. Later, Derksen and the first author~\cite{DK} 
proved that there are uncountably many non-equivalent embeddings of $\C$ into 
$\C^n$ for $n>1$. Their result heavily uses the fact that the holomorphic 
automorphism group of $\C$ is a Lie group, i.e., it is very small in comparison 
to the automorphism group of complex Euclidean spaces in  dimensions greater than $1$.
Combining the ideas of that paper with the cancellation property for Eisenman hyperbolic spaces
Borell and the first author then proved (see \cite{BK}) that the number of equivalence 
classes of proper holomorphic embeddings of $\C^k$ into $\C^n$ is uncountable 
for any $0<k<n$. The last two above mentioned results are proven by using the Cantor diagonal process and
it remained still an unsolved challenging problem whether non-equivalent embeddings could occur in continuous
or even holomorphic families. 

Our first main result gives an affirmative answer.

\begin{theorem}\label{main}
Let $X$ be a complex space, which can be embedded into $\C^n$ and
such that the group of holomorphic automorphisms
$\aut_{\rm{hol}}(X)$ is a Lie group. Then there exist, for
$k=n-1-\dim X$,  a family of holomorphic embeddings of $X$ into
$\C^n$ parametrized by $\C^k$, such that for different parameters
$w_1\neq w_2\in \C^k$ the embeddings
$\varphi_{w_1},\varphi_{w_2}:X \hookrightarrow \C^n$ are
non-equivalent.
\end{theorem}

\begin{notation}
Observe that for $k=0$, $\dim X=n-1$ the conclusion of the theorem
is empty. In this situation it is still known that there are uncountably many equivalence classes
of embeddings by the above  mentioned results from \cite{BK}.
\end{notation}

We would like to emphasize that there is another (weaker) definition of 
equivalence --- called $\Aut(\C^n)$-equivalence --- which is used by several 
authors, e.g., Buzzard, Forstneri\v c, Globevnik and Varolin. In these papers 
uncountability of certain equivalence classes of embeddings in this weaker 
sense is proved. Our main result is much stronger than these results. In our 
definition, two embeddings $\Phi,\Psi\colon X\hookrightarrow\C^n$ are equivalent if 
their {\bf images} coincide modulo $\Aut(\C^n)$, i.e., if there is an automorphism 
$\varphi\in\Aut(\C^n)$ such that the  images of $\varphi\circ\Phi$ and $\Psi$ 
coincide. In such a situation, the map $\Psi^{-1}\circ\varphi\circ\Phi$ is well 
defined and it is an automorphism of $X$. The weaker notion mentioned above 
demand that $\varphi\circ\Phi$ and $\Psi$ are equal as maps, i.e., it demands 
that $\Psi^{-1}\circ\varphi\circ\Phi$ is the identity on $X$.  Our application 
to group actions would not work for the weaker definition.

Using the cancellation property for Eisenman hyperbolic spaces we can cross our situation with some
affine space and we are able to conclude

\begin{theorem}  \label{main1}(see Corollary \ref{main corollary})
There exist, for $k=n-l-1$,  a family of holomorphic embeddings of
$\C^l$ into $\C^n$ parametrized by $\C^k$, such that for
different parameters $w_1\neq w_2\in \C^k$ the embeddings
$\psi_{w_1},\psi_{w_2}:\C^l \hookrightarrow \C^{n}$ are
non-equivalent.
\end{theorem}

We also give an application of Theorem \ref{main} to actions of compact (or equivalently complex
reductive, see \cite{Ku}) groups on $\C^n$. It was a long standing problem whether all holomorphic actions of such groups on affine space are linear after a change of variables (see for example the overview article \cite{Hu}). The first counterexamples to that (Holomorphic Linearization) problem were constructed by Derksen and the first author in \cite{DK1}. In the present paper we show that the method from there is holomorphic in a parameter and therefore  applied to our parametrized situation leads to

\begin{theorem} \label{action1}
For any $n\ge 5$ there is a holomorphic family of $\C^*$-actions on $\C^n$ parametrized by $\C^{n-4}$
$$\C^{n-4} \times \C^* \times \C^n \to \C^n \quad(w, \theta, z) \mapsto \theta_w (z)$$ so that for
different parameters $w_1\neq w_2\in \C^{n-4}$ there is no equivariant isomorphism between the actions $\theta_{w_1}$ and  $\theta_{w_2}$.
\end{theorem}

The linearization problem for holomorphic $\C^*$-actions on $\C^n$ is thus solved to the positive for $n=2$ by Suzuki \cite{Su} and still open for $n=3$. For $n=4$ there are uncountably many actions (non-linearizable ones among them)
\cite{DK} and for $n\ge 5$ our result implies that there are families. Moreover there are families including a linear action as a single member of the family as our last main result shows

\begin{theorem} \label{action2}
For any $n\ge 5$ there is a holomorphic family of $\C^*$-actions on $\C^n$ parametrized by $\C$
$$\C \times \C^* \times \C^n \to \C^n\quad (w, \theta, z) \mapsto \theta_w (z)$$ so that for
different parameters $w_1\neq w_2\in \C$ there is no equivariant isomorphism between the actions
$\theta_{w_1}$ and $\theta_{w_2}$. Moreover the action $\theta_0$ is linear.
\end{theorem}

The paper is organized as follows. In section  \ref{technik} we give all technical preparations for our
(quite complicated) construction. The proofs will be given in the appendix. Section \ref{hauptsatz} contains the proof of  Theorem \ref{main}.  The next section \ref{Eisenman} contains an addition to Theorem \ref{main} which allows to deduce Theorem \ref{main1}.
Section \ref{wirkung} contains the application to group actions in particular the proofs  of Theorems \ref{action1} and  \ref{action2}. Some concluding remarks are contained in section \ref{concluding}.

The results of the present paper have been partially announced in \cite{K1}. At that time the technical details had been extremely complicated and lengthy. Over the last years they have become much shorter and much more elegant so that the authors finally decided to publish the present complete version.

Part of the work was done during a stay of the first author at the Mittag-Leffler-Institute during the special program in Complex Analysis of Several Variables 2008. We would like to thank the Institute for hospitality and excellent working conditions. 

\section{Technical preparations}\label{technik}
In this section we state the main lemmas needed to prove the main theorem. The proofs of theese lemmas can be found in section \ref{prooflemmas}.

For the benefit of the reader  we give
 a list of notations  mostly adhered to in this paper.

By an automorphism of $\C^n$ depending on a
parameter $w\in \C^k$ we mean an element of
$\mathrm{Aut}^{k}_{hol}(\C^{n}):=  \{ \psi \in \mathrm{Aut}_{hol}(\C^{k+n})\quad :
\quad \psi (w, z)= (w, \psi_1 (w,z))\} $, and approximations are understood
to be uniform on compacts.

We will throughout the paper call a holomorphic map $\eta : \C^k \to \C^n$
a {\em parametrized point} $\eta (w)$ in $\C^n$.

\begin{itemize}
\item $\B_i$ is the open unit ball of $\C^i$. The closed ball is denoted $\overline \B_i$.
\item $X$ a complex space of dimension $\dim X$. 
\item $\varphi_0=\iota :X\hookrightarrow \C^n$, where $\iota$ is the inclusion map.
\item $\phi_0:\C^k \times X \to \C^k\times \C^n$ is given by $(w,x)\mapsto (w,\varphi_0(x))$ for $w\in \C^k$.
\item $\alpha_n\in \rm{Aut}^k_{hol}(\C^n)$.
\item $A_n=\alpha_n\circ \alpha_{n-1}\circ \ldots \circ \alpha_1$.
\item $\phi_n=A_n\circ \phi_0=\alpha_n\circ \alpha_{n-1}\circ \ldots \circ \alpha_1\circ \phi_0:\C^k \times X \to \C^k\times \C^n$.
\item $\pi_2$ the projection of $\C^k\times \C^n$ onto $\C^n$.
\item $\phi =\lim_{n\to \infty} A_n\circ \phi_0(w,x)$.
\item $\varphi_n=\pi_2(A_n\circ \phi_0)=\pi_2(\phi_n)$.
\item $\varphi =\pi_2(A\circ \phi_0)$ (or $\varphi_w$ if $w$ is a fixed parameter value).
\item $P_2:\C^k\times X\to X$ is given by $P_2(w,x)=x$.
\item $\xi_i(w)\in \C^n$ interpolation points (osculation points) which vary with respect to $w\in \C^k$.
\item $\eta_i$ the points of $X$  corresponding to the points  $\xi_i(w)$ (preimage points of $\xi_i(w)$).
\item $\mu$ induction variable.
For every $\mu$ we define $\epsilon_{\mu},R_{\mu}>0$ and finite subsets $\cup_{j=1}^{k(\mu )}\{a^{\mu}_j\}$ of $\partial (\mu +1)\B_n$ and $\cup_{j=1}^{k(\mu )}\{x^{\mu}_j\}$ of $X$ respectively.
\end{itemize}

\subsection{Growth restrictions for holomorphic maps}

In the construction of our families of non-equivalent embeddings we
will use techniques of growth restrictions on entire maps from
$\C^n$ to $\C^n$. 

These growth restrictions are  governed by the following
lemma, Lem\-ma 4.3 in \cite{RR1} which we present with a simple additional conclusion,
namely that one can avoid the nowhere dense set $Q$. The additional conclusion is obvious from the proof.

\begin{lemma}\label{lemma4.3copy}
Given real numbers $0<a_1<a_2$, $0<r_1<r_2$ and $c>0$, let
$\Gamma$ be the class of holomorphic mappings
$$f=(f_1,\ldots ,f_k):a_2\B_n\to r_2\B_k$$
such that
$$|f(0)|\leq \frac 12r_1$$
and
$$\| \frac{\partial (f_1,\ldots ,f_k)}{\partial (z_1,\ldots ,z_k)}\| \geq c$$
at some point of $a_1\bar \B_n$.

Let $Q\subset \partial (r_1\B_k )$ be a set such that $\partial
(r_1\B_k ) \setminus Q$ is dense in $\partial (r_1\B_k).$ Then there
is a finite set $E=E(a_1,a_2,r_1,r_2,c)\subset
\partial (r_1\B_k )\setminus Q$ with the property that, if $f\in \Gamma$ and
$f(a_1\B_n)$ intersects $\partial (r_1\B_k )$ then $f(a_2\B_n)$
intersects $E$.
\end{lemma}

The following  technical detail is well known, we include for completeness, it is
Lemma 5.4. in \cite{F}. It  will be used frequently
in the proof of Theorem \ref{main}.

\begin{lemma}\label{FGRlemma2copy}
Let $K$ be a polynomially convex set in $\C^n$ and let $X$ be a closed
analytic  subvariety of $\C^n$. Moreover let $X_0$ be a compact holomorphically convex subset of $X$,
such that $K \cap X_0$ is contained in the (relative) interior of $X_0$.

Then the set  $K\cup X_0$ is polynomially convex.

\end{lemma}

We  need to construct  parametrized points with a certain property.

\begin{proposition}\label{pointprop}
Given natural numbers $k>0$ and $n\geq 2$ there is a number $m\in
\N$ such there are $m$, pairwise different,  parametrized points $\xi_1,\xi_2,
\ldots ,\xi_m:\C^k \to \C^n$ parametrized by a parameter $w\in
\C^k$ with the following property:

For $w_1\neq w_2$ there is no affine automorphism $\alpha \in
\mathrm{Aff}(\C^n)$ which maps the set of points $\{
\xi_1(w_1),\xi_2(w_1),\ldots ,\xi_m(w_1)\}$ onto the set of points
$\{ \xi_1(w_2),\xi_2(w_2),\ldots ,\xi_m(w_2)\}$.
\end{proposition}

It is not difficult to see that the  holomorphic (even the algebraic)
automorphism group of $\C^n$ ($n\geq 2$) acts transitively on
finite subsets of $\C^n$ with fixed cardinality, i.e. one can move
$k$ distinct points $z_1,z_2,\ldots ,z_k\in \C^n$ by an
automorphism into some standard position, for example to the
points $(1,0,\ldots , 0)$,$(2,0,\ldots , 0)$,$\ldots$,$(k,0,\ldots
, 0)$, (usually if the dimension is clear we will write $(k,0)$).
For this apply a generic linear change of coordinates so that
afterwards all coordinates of the points $ z_1,z_2,\ldots ,z_k$
become different and then apply appropriate shears. The question
whether the holomorphic automorphism group of $\C^n$ acts
transitively on countable discrete subsets of $\C^n$ was answered
to the negative by Rosay and Rudin in \cite{RR1}. They called the
countable discrete subsets in the $\mathrm{Aut}_{hol}(\C^n
)$-orbit the "standard" countable discrete subset $e_1\N =\{
(1,0),(2,0),\ldots ,(k,0),\ldots \}$ tame sets and proved the
existence of non tame sets.

For our construction of families of embeddings we have to move
finitely many points holomorphically depending on a parameter to
some prescribed position by an automorphism which also depends
holomorphically on the parameter. 

%

\begin{definition}
Let $N$ be a natural number and $\zeta_1,\zeta_2,\ldots
,\zeta_N:\C^k\to \C^n$ be holomorphic maps such that for each
fixed parameter $w\in \C^k$ the $N$ points
$\zeta_1(w),\zeta_2(w),\ldots ,\zeta_N(w)$ in $\C^n$ are
different. We call them \textbf{simultaneously  standardizable},
if there exists an automorphism $\psi  \in
\mathrm{Aut}^k_{hol}(\C^n)$ holomorphically depending on the
parameter $w$ with
$$\psi (w,\zeta_i(w))=(w,(i,0)) \quad \textrm{for all } i=1,2,\ldots ,N \quad \textrm{and for all } w\in \C^k .$$
\end{definition}

 At the moment we are not able to decide under which condition on the parameter space any collection of parametrized points is
 simultaneously standardizable. We have the following partial result which is 
sufficient for our purposes.

\begin{proposition}\label{standp}
Let $\zeta_1,\zeta_2,\ldots ,\zeta_N:\C^k\to \C^n$, $n\geq 2$, be
holomorphic maps such that the points
$\zeta_1(w),\zeta_2(w),\ldots ,\zeta_N(w)$ in $\C^n$ are different
for each $w\in \C^k$ and suppose $k<n-1$. Then the parametrized
points $\zeta_1,\zeta_2,\ldots ,\zeta_N$ are simultaneously
standardizable.
\end{proposition}

\subsection{Interpolation lemma}
The following lemma is a key ingredient in the proof of our main
theorem. It can be used to prove interpolation results for parametrized embeddings as this is well-known in the non parametrized version. For the non-parametrized case we refer to the papers \cite{G}, \cite{K},
\cite{FIKP}, \cite{F}. Another approach to interpolation is used in \cite{Pr1}. We do not prove parametrized interpolation theorems  for countable sets in this paper, we just use the lemma as a tool in the proof of our main Theorem \ref{main}.

Therefore before we formulate the lemma we need to introduce some
notation:


Let  $\phi : \C^k\times X\hookrightarrow \C^k \times  \C^n$ be a
(parametrized) embedding of a complex space $X$ into $\C^n$, i.e., an embedding of the form 
$\phi (w, x) = (w, \tilde \phi (w, x))$.

$X_R$ is a holomorphically convex compact subset of $X$.

$\overline \Delta$ is a ball (of  any radius) in $\C^k$.
$\overline \B$ is a ball (of any radius) in $\C^n$.

We assume that $\phi^{-1}_w(\overline \B )\subset X_R$ for all
$w\in \overline \Delta$, which implies that $K=(\overline \Delta
\times \overline \B )\cup \phi (\overline \Delta \times \overline
X_R)$ is a polynomially convex subset of $\C^k\times \C^n$ (Lemma \ref{FGRlemma2copy} above).

Furthermore we assume the dimension condition
\begin{equation}\label{dim}
 \dim X +k <n .
\end{equation}
\begin{lemma}\label{FLYTTLEMMAT}
Let $b_1(w),b_2(w),\ldots , b_N(w)$ be $N$ parametrized points
contained in $K=(  \overline \Delta \times  \overline \B )\cup
\phi ( \overline \Delta \times \overline X_{R})$. Assume that
$(w,p(w))$ and $(w,q(w))$ are parametrized points in  $\C^k \times
\C^n\setminus K $ and  let $s$ be some positive integer. Then for
each $\epsilon >0$ there exists an automorphism  $\alpha \in
\rm{Aut}^k_{hol}(\C^n)$ such that $|\alpha (w,z)-(w,z)|\leq
\epsilon$ for every $(w,z)\in K$, $\alpha (w,b_i(w) )=(w,b_i(w))$
of order $s$ for every $w\in \C^k$ and $\alpha (w,p(w))=(w,q(w))$
for every $w\in \C^k$.
\end{lemma}

\subsection{Osculation lemma}

The conclusion of the growth restrictions will be that only affine automorphisms could map the images of different embeddings from the family  onto each other. To exclude these affine automorphisms we will have a finite number of points "marked" in each embedding of our family. 
The marking is in such a way that by affine automorphisms the  sets of marked points  have to be mapped onto each other. The marking will be achieved by letting the embeddings osculate of higher order exactly at these points.

\begin{definition}
Let $l\geq 2$ be a natural number. We say that a submanifold $M$
of $\C^n$ osculates of order $l$ at $x\in M$ if $M$ has contact
order $l$ with the tangent space $T_xM\subset \C^n$ in $x \in M$.
\end{definition}

In local coordinates osculating can be interpreted as follows:

Let $\zeta :U(\subset \C^m)\to M$ be a holomorphic coordinate
system for the $m$ dimensional manifold $M$ at $x$, $\zeta (0)=x$.
Then $M$ osculates of order $l$ at $x$ if and only if
$\frac{\partial}{\partial w^{\alpha}}|_{w=0}\zeta \in T_xM$ for
every multiindex $\alpha =(\alpha_1,\alpha_2,\ldots ,\alpha_m)$
with $2\leq |\alpha |\leq l$.

The property to osculate is preserved by affine coordinate changes
on $\C^n$, i.e. if $\psi :\C^n \to \C^n$ is an affine automorphism
of $\C^n$ then the submanifold $M\subset \C^n$ osculates of order
$l$ at $x\in M$ if and only if the submanifold $\psi (M)$
osculates of order $l$ at $\psi (x)\in \psi (M)$.

\begin{notation}\label{oscul}
(1) The property not to osculate of order $l$ at any point is generic for $l\ge 2$ if not $\dim M =1$
and $n=2$. In the later case it is generic for $l\ge3$.

(2) In the proof of Theorem \ref{main} the role of $M$ will be played by the smooth part $X \setminus Sing (X)$ of the space we want to embed.
\end{notation}

By a manifold $M_{\C^k}$ in $\C^n$ parametrized by $w\in \C^k$ we mean the image (in $\C^k \times \C^n$) of a proper holomorphic
embedding $\Phi : \C^k \times M \to \C^k\times \C^n$ which is of form $(w,m) \mapsto (w, \phi (w, m))$. By $M(w)$ we denote  the image $\Phi (\{w\}\times M) \subset \C^n$.

\begin{lemma}\label{KYSSLEMMAT}
Let $M_{\C^k}$ be a manifold in $\C^n$
parametrized by $w\in \C^k$ and assume that $m=\dim M <n$. Let $\xi_1(w),\ldots, \xi_t(w)$ be
simultaneously standardizable parametrized points such that
$\xi_i(w)\in M(w)$ for $i=1,\ldots , t$. Then for $l \ge 2 \in \N$ there
exists an automorphism $\kappa \in \rm{Aut}_{hol}^k(\C^n)$ such that
$\kappa (M(w))$ osculates of order $l$ in $\xi_i(w)$ for
$i=1,\ldots, t$ and all $w\in \C^k$.

\end{lemma}

\subsection{Non-osculation lemma}
To keep the osculation of order $l$ in $\xi_i(w)$ and make sure that the embedding do not osculates of order $l$ at other points, we use the following lemma.

\begin{lemma}\label{AKlemma}
Let $M_{\C^k}$ be a manifold in $\C^n$ para\-metrized
by $w\in \C^k$ and denote $m = \dim M$. Also let
\begin{enumerate}[a)]
\item $K_M$ be a compact subset of $\C^k\times M_{\C^k}$. 

\item $K$ be
a compact subset of $\C^k\times \C^n$. \item $a_1,a_2,\ldots ,a_r$
be finitely many points in $K_M$. 
\item $b_1(w),b_2(w),\ldots
,b_q(w)$ be finitely many parametrized points in $\C^k\times
M(w)\setminus K_M$, (these are the points where we
want to keep the osculation order $l$). \item $l\geq 2$ be a natural
number. If $m=1$ and $n=2$ let $l\geq 3$. \item $\epsilon >0$ be a
real number.
\end{enumerate}
Then there exists an automorphism $\psi \in \rm{Aut}_{hol}^k(\C^n)$ such
that
\begin{enumerate}[1.]
\item $\psi (a_i)=a_i$ for every $i=1,2,\ldots ,r$. \item
$\psi_w(z)=z+O(|z-b_i(w)|^{l+1})$ as $z\to b_i$ for every
$i=1,2,\ldots ,q$ and every $w\in \C^k$. \item
$|\psi_w(z)-z|+|\psi_w^{-1}(z)-z|<\epsilon$ for every $(w,z)\in K$
\item There is no point $x\in K_M$ such that $\psi (M(w))$
osculates of order $l$ in $\psi (x)$.
\end{enumerate}
\end{lemma}

\section{Proof of main theorem.} \label{hauptsatz}

To be precise let us define the notion used in the formulation of Theorem \ref{main}

\begin{definition} 
Let X, Y be complex spaces, Z a complex (resp. topological) space.
A holomorphic (resp. continuous) map
$$\Phi : Z \times X \to Y$$ is called a {\sl holomorphic (resp. continuous) family of 
holomorphic embeddings of $X$ into $Y$ parametrized by (a parameter) in $Z$} if for each point
$z$ in the parameter space $Z$ the map $\Phi_z : X\to Y, \ \ x\mapsto \Phi(z,x)$ 
is  a proper holomorphic embedding. 
\end{definition}

In the proof we are working with families of embeddings of some complex space $X$
into $\C^n$ parametrized by $\C^k$ which come from the following construction.

\begin{notation} 
If a holomorphic map $$\phi : \C^k \times X \to \C^k \times \C^n\label{??.1}$$ 
of the form

\begin{equation}
\phi (w, x) = (w, \tilde \phi (w, x)), \ \ w\in \C^k,\ x\in X  \label{wei} 
\end{equation}

is a proper holomorphic embedding of $\C^k\times X$ into $\C^k \times \C^n$ then
the map $$\tilde \phi : \C^k \times X \to \C^n$$ (where $\tilde \phi$ is defined by
(\ref{wei}))
is a family of holomorphic embeddings of
$X$ into $\C^n$ parametrized by $\C^k$. Note that the contrary does not hold, i.e.,
if $$\tilde \phi : \C^k \times X \to \C^n$$ is a family of holomorphic embeddings of 
$X$ into $\C^n$ parametrized by $\C^k$, the corresponding map 
$$ \phi : \C^k \times X \to \C^k \times \C^n$$ defined by 
$$\phi (w, x) = (w, \tilde \phi (w, x)), \ \ w\in \C^k,\ x\in X$$ may fail to
be an embedding (see the example below).

Also note the following fact which we will use in our construction :

\smallskip\noindent
If $\alpha \in \aut_{hol} (\C^k\times \C^n)$ is of the form $\alpha (w,z) = (w, \tilde \alpha (w,z))$, i.e. $\alpha \in \Aut_{hol}^k (\C^n)$, then $\alpha \circ\ \phi$
is again an embedding of the form (\ref{wei}), hence its "second coordinate"
$\pi_2 \circ (\alpha \circ \phi): \C^k \times X \to \C^n$ is a
family of holomorphic embeddings
(where $\pi_2 : \C^k \times \C^n \to \C^n$ is defined by $(w, v) \mapsto v, \ \ w\in \C^k,
v \in \C^n$).
\label{param}
\end{notation}

\begin{example} 
If $$\tilde \phi : \C^k \times X \to \C^n$$ is a holomorphic family of holomorphic embeddings of 
$X$ into $\C^n$ parametrized by $\C^k$ then it is straightforward to prove that
the map $$\phi : \C^k \times X \to \C^k \times \C^n$$ defined by
$$\phi (w, x) = (w, \tilde \phi (w, x)), \ \ w\in \C^k,\ x\in X$$
is holomorphic, injective and immersive. On the other hand properness may fail,
as the following example shows:
\smallskip\noindent
We are going to define a holomorphic family of embeddings of $\C$ into $\C^2$
parametrized by $\C$.
Define $f : \C \times \C \to \C^2$ by $f(y, x)= (x + y\cdot x^2, y\cdot x)$.
For each fixed point $y$ in the parameter space $\C$ we
are given a proper, injective, immersive, holomorphic map from $\C$ into $\C^2$
(for $y\ne 0$ the second coordinate by itself gives already such an embedding, 
and for $y=0$ the first coordinate is such an imbedding). 
\par
On the other hand the map $\C \times \C \to \C\times \C^2$ defined by
$(y, x)\mapsto (y, x + y\cdot x^2, y\cdot x)$ is not proper. Indeed the sequence
$(x_n, y_n)$ defined by $x_n = n$ and $y_n = {1-n \over n^2}$ leaves any compact
subset of the definition space but is mapped onto the sequence $({1-n \over n^2}, 
1, {1-n\over n})$ which converges to the point $(0,1, -1)$ in the target space
(which is not in the image, in fact the image is not closed it is the hypersurface
$\{(a,b,c)\in \C^3 : a \cdot b =c (c+1) \}$ except the line $\{ c=-1, a=0 \}$).
\end{example}

We would like to emphasize that we will prove a slightly stronger statement than just holomorphic 
families of embeddings. Our families are always such that the map $\Phi : \C^k \times X \to \C^k \times \C^n$ is a proper holomorphic embedding which we will denote by using the symbol $\hookrightarrow$, i.e.  $\Phi : \C^k \times X \hookrightarrow \C^k \times \C^n$ for our families. This subtle point plays a role in the last section, since if $\Phi (\C^k\times X)$ is not a closed submanifold in $\C^k\times \C^n$ the
construction of pseudo-affine modification does not work.

\begin{notation}
In the proof of Theorem \ref{main} we use the property that the group 
of holomorphic automorphisms of $X$ can be exhausted by a sequence of compact 
subsets (in c.-o. topology). A Lie group (with possibly countably many 
components) can be exhausted by a sequence of compact subsets. On the other 
hand, if a locally compact topological group acts effectively on a manifold it 
is a Lie group~\cite{BM}. Since the group of holomorphic automorphisms of a 
Stein space $X$ acts effectively on the smooth part $\tilde X = X\setminus\Sing(X)$, this 
implies that $\Aut(X)$ is a Lie group if and only if it can be exhausted by 
compacts.
\end{notation}

\begin{proof} (of Theorem \ref{main})

By assumption the complex space $X$ embeds into  $\C^n$, say
$\varphi_0 : X\hookrightarrow \C^n$.   We 
start with the trivial family of embeddings of $X$ into $\C^n$, $\phi_0 :\C^k \times X\to \C^k\times
\C^n$ given by $(w,x)\mapsto (w,\varphi_0(x))$.

We will construct automorphisms $\alpha_n\in
\mathrm{Aut}^k_{hol}(\C^n)$ recursively. Let $A_n=
\alpha_n\circ\alpha_{n-1}\circ \ldots \circ \alpha_1$.
We further arrange $\alpha_n\in  \mathrm{Aut}^k_{hol}(\C^n)$ inductively such that
$$A(w,z)=\lim_{n\to \infty}A_n(w,z)=(w,\lim_{n\to \infty}\tilde A_n(w,z))$$
exist uniformly on compacts for some open neighborhood $\Omega
\subset \C^k \times \C^n$ containing $\C^k\times X$ and such that
the mapping $A:\Omega \to \C^k\times \C^n$ given by $(w,z)\mapsto
(w,\tilde A_n(w,z))$ defines a biholomorphic mapping on
$\C^k\times \C^n$. The existence of the limit follows from
Proposition 4.1 and 4.2 in \cite{F}.

Now let $\phi_n=\alpha_n\circ\alpha_{n-1}\circ \ldots \circ
\alpha_1\circ \phi_0:\C^k\times X\to \C^k\times \C^n$. Also define
$\pi_2$ as the projection of $\C^k\times \C^n$ to $\C^n$.

The family of holomorphic embeddings will then be given by the
second coordinate of $A\circ \phi_0=\lim_{n\to \infty}\phi_n$,
i.e. $\varphi =\pi_2(A\circ \phi_0)$.

It follows that $\phi=\lim_{n\to \infty}A_n\circ
\phi_0=A\circ \phi_0$ is a proper holomorphic embedding of $\C^k
\times  X$ into $\C^k\times \C^n$, which gives that $\varphi
=\pi_2(A\circ \phi_0)$ is a family of holomorphic embeddings of
$X$ into $\C^n$ parametrized by $\C^k$.

In order to  make the embeddings $\phi_w$  different for
different parameters, we will choose them such that no affine
automorphism,  can map the image of one of these embeddings onto
another.   At the same time we make sure, using growth 
conditions, that the only way to map the image of one embedding
onto another is by an affine mapping. This ensures 
that the embeddings will be non-equivalent  for different parameters $w$.

So the construction in short is:
\begin{enumerate}[a)]
\item Choose sufficiently many points $\xi_1(w), \ldots
,\xi_m(w)$ in correct positions such that no affine
automorphism can map the points $\{ \xi_i(w_1)\}$ into the points
$\{ \xi_i(w_2)\}$ for $w_1\neq w_2$. These points will be chosen
differently for each parameter (however holomorphically depending
on the parameter). 
\item Embed the space $X$ through these points.
\item In order to single out these points we make sure that our
embeddings osculate at these points of a certain order $l$, and
osculates of order less than $l$ at all other points. 
\item Divide the rest of
$\C^n$ into concentric shells with increasing radii. In each shell
we choose inductively points through which we later will embed the space $X$. These
points and their preimages in $X$ are inductively chosen in such a way that we
get some growth conditions on the embedding. 
\item Embed $X$ inductively through all the points with careful chosen preimage points. In each
step of the inductive process $X$ will be embedded through the (finitely many) points contained in one shell, keeping the points from previous shells contained in the image of $X$.
\item We
then show, using the growth conditions introduced in d), that for
two embeddings to be equivalent for different parameters they have
to differ by an affine automorphism.
 \item The condition in a)
shows that there is no such affine automorphism. Consequently 
the embeddings are different for all parameters.
\end{enumerate}

The first automorphism $\alpha_1$ will take care of that for
different parameters $w_1\neq w_2\in \C^k$ there is no affine
automorphism $\beta \in \textrm{Aff}(\C^n)$ mapping the image of
$\varphi_{w_1}(X)$ onto the image of $\varphi_{w_2}(X)$. For this
first use Proposition \ref{pointprop} to get $m$ points
$\xi_1,\ldots \xi_m:\C^k \to \C^n$ parametrized by an parameter.
Then choose $m$ points, $\eta_1,\ldots ,\eta_m$ in the smooth part $\tilde X$ of $X$ or more exactly of
$\phi_0 (X)\subset \C^n$. After that use Lemma \ref{FLYTTLEMMAT} to
find an automorphism $\gamma \in \mathrm{Aut}^k_{hol}(\C^n)$ such
that $\gamma (w,\eta_i)=(w,\xi_i(w))$ for every $w\in \C^k$.

Fix a natural number $l\ge 2$ such that not osculating of order
$l$, at any point, is a generic property for a submanifold of
$\dim X$ in $\C^n$ (see Remark \ref{oscul}). Using Lemma \ref{KYSSLEMMAT} and Lemma
\ref{AKlemma} we get an automorphism $\delta \in
\mathrm{Aut}^k_{hol}(\C^n)$ with $\delta
(w,\xi_i(w))=(w,\xi_i(w))$ for every $w\in \C^k$ that prescribes
the higher derivatives of $\delta$ in the $\C^n$-direction such
that for all $w\in \C^k$ the subvariety $\delta (\gamma
(w,\varphi_0(X)))$ of $\C^n$ will be tangent of order $l$ at the
points $\xi_i(w)$.

The automorphism $\alpha_1\in \mathrm{Aut}^k_{hol}(\C^n)$ is now
given by $\alpha_1=\delta \circ \gamma$. We also note that with
$\alpha_1$ we have the osculation order in the points we want, so in the
future we only have to make sure we do not destroy the osculation order
in $\xi_i(w)$ but destroy it everywhere else.

We now choose an exhaustion of the space $T=\aut (X)$ with compact
sets $T_i$ such that  $$T=\cup_{i=1}^{\infty}T_i\textrm{  and
}T_i\subset \overset{\circ}T_{i+1}.$$ Further let $\rho :X\to
\R^{\geq 0}$ be a continuous exhaustion function of $X$, so
$X_r:=\rho^{-1}([0,r])$ is a compact subset of $X$ for every
$r\geq 0$. ($X_r=\{ x\in X:\rho (x)<r \} \subset \subset X$ for
every $r$, for example $\rho (x)=\| \iota (x)\|^2$ will work.)

Denote the unit ball in $\C^n$ by $\mathbb B_n=\{ z\in \C^n:\| z\|
<1\}$ and $\B_k =\{ w\in \C^k:\| w\| <1\}$. Choose a sequence of
relatively open neighborhoods $U_i$, $i=1,2,3,\ldots$ of the set
$\eta =\cup_{i=1}^m\{ \eta_i\}$ in $X$ with
$\cap_{i=1}^{\infty}U_i=\eta$ and $\overline U_{i+1}\subset
\overset{\circ}U_i$.

Remember that the points $\eta_i$ in $X$ are the preimages of the
points in $\C^n$ at which the varieties $\varphi (w,X)$ have
osculation of order $l$, i.e. the points $\xi_i(w)$, for every
$w\in \C^k$.

Now we inductively, for $\mu =1,2,3,\ldots$, define real numbers
$\epsilon_{\mu}, R_{\mu}>0$,  finite subsets $\cup_{j=1}^{k(\mu
)}\{ a^{\mu}_j\}$ of $\partial (\mu +1)\B_n$ and finite subsets
$\cup_{j=1}^{k(\mu )}\{ x^{\mu}_j\}$ of $X$ with the same
cardinality $k(\mu )$, and automorphisms $\alpha_{\mu}\in
\mathrm{Aut}^k_{hol}(\C^n)$ of $\C^n$ parametrized by $w\in \C^k$.
When choosing the subsets $\cup_{j=1}^{k(\mu )}\{ a^{\mu}_j\}$ and
$\cup_{j=1}^{k(\mu )}\{ x^{\mu}_j\}$ it is important to remember
that since we are going to embed the point $x^{\mu}_j$ of $X$
through $a^{\mu}_j$, we have to choose the points $a^{\mu}_j$ such
that $\cup_{j=1}^{k(\mu )}\{ a^{\mu}_j\} \cap
\pi_2(\phi_{\mu}(\C^k\times X))=\emptyset$. This is possible by
Lemma \ref{lemma4.3copy} due to the dimension of $X$, since
$\pi_2(\phi_{\mu}(\C^k\times X))\cap \partial (\mu +1)\B_n$ has
measure zero in $\partial (\mu +1)\B_n$.

Start by letting $\epsilon_1=1, R_1=1, k(1)=m$ and $\alpha_1$ be
as constructed earlier. For $\mu \geq 2$ we construct these
entities such that the following properties are satisfied:

\begin{enumerate}[$1_{\mu}$.]
\item $0<\epsilon_{\mu}<\frac{\epsilon_{\mu -1}}{3}$ \item If
$F:\B_n\to (\mu +2)\B_n\setminus \cup_{j=1}^{k(\mu )}\{
a^{\mu}_j\}$ is a holomorphic mapping with $\| F(0)\| \leq
\frac{\mu +1}2$ and $| JF(0)|\geq 1$ then
$F((1-\frac{\epsilon_{\mu}}{2})\B_n)\subset (\mu +1)\B_n$.

\item $\phi_{\mu}(w,x^{\mu}_j)=\alpha_{\mu}\circ \phi_{\mu
-1}(w,x^{\mu}_j)=(w,a^{\mu}_j)$ and $\rho (x^{\mu}_j)>\max \{ \rho
(t(x)):t\in T_{\mu}, x\in P_2(\phi_{\mu -1}^{-1}(\mu \overline
\B_k\times  \mu \overline \B_n))\}$ where  $P_2:\C^k \times X\to
X$ is given by $P_2(w,x)=x$.

\item  $\| \alpha_{\mu}\circ \phi_{\mu -1}(w,x)-\phi_{\mu
-1}(w,x)\| =\| \phi_{\mu }(w,x)-\phi_{\mu -1}(w,x)\| \leq
\epsilon_{\mu}$ for every $x\in X_{R_{\mu -1}}$ and for every
$w\in (\mu -1)\overline \B_k$.

\item $\| \alpha_{\mu} (w,z)-(w,z)\| \leq \epsilon_{\mu}$ for
$(w,z)\in \mu \overline \B_k \times \mu \overline \B$

\begin{notation}
Motivated by $4_{\mu}$ and $5_{\mu}$ we define the compact set
$$K_{\mu}=(\mu \overline \B_k \times \mu \overline \B_n)\cup \phi_{\mu -1}( (\mu -1) \overline \B_k \times \overline X_{R_{\mu -1}}).$$
\end{notation}

\item $\alpha_{\mu}\circ \phi_{\mu -1}(w,x^l_j)=(w,a_j^l)$ for
every $w\in \C^k$  and for every $l<\mu$, $j=1,2,\ldots ,k(l)$.

\item $\alpha_{\mu}(w,z)=(w,z)+O(|z-\xi_i(w)|^{l+1})$ as $z\to
\xi_i(w)$.

\item For fix $w\in \mu \overline \B_k$ the submanifold
$\pi_2(\phi_{\mu}(w,\tilde X))$ of $\C^n$ do not osculate of order
$l$ in any point $\pi_2(\phi_{\mu}(w,x))$ with $x\in (X_{R_{\mu
-1}}\cap \tilde X )\setminus U_{\mu}$. Here $\tilde X$ is the
union of all components of $X$ which are smooth and of maximal
dimension.

\item $\| \phi_{\mu}(w,x)\| \geq \mu +1$ for every $x\in
X\setminus X_{R_{\mu }}$ and for every $w\in \mu \overline
\B_k$.

\item $R_{\mu  }>R_{\mu -1}+1$.

\end{enumerate}

We will now confirm that such a construction is possible.

For step 2 of the induction we choose, in the following order,

\begin{enumerate}[(1)]
\item $\epsilon_2<\frac{\epsilon_1}{3}$ \item $\cup_{j=1}^{k(2)}\{
a^2_j\}$ a finite subset of $\partial (3\B_n)\subset \C^n$ which
does not intersect the image $\pi_2(\phi_1 (\C^k\times X))$ and
satisfies $2_2$. This is possible by Lemma \ref{lemma4.3copy}, namely
$\dim X+k<n$ makes it possible to choose the points $\{ a^2_j\}$
outside $\pi_2(\phi_1 (\C^k\times X))$ (which has measure zero by Sards theorem).

\item $\cup_{j=1}^{k(2)}\{ x^2_j\}$ a finite subset of $X$ such
that $\rho (x_j^2)>\max \{ \rho (t(x)): t\in T_2, x\in
P_2(\phi_1^{-1}(2\overline \B_k \times 2\overline \B_n))\}$.
Property $3_2$ will then be fulfilled. We shall also choose
$x_j^2$ such that $(w,x_j^2)\notin \hat K_2$ for every $j$, where
$K_2=(2\overline \B_k \times 2\overline \B_n)\cup \phi_1(
1\overline \B_k \times \overline X_{R_1})$ is a compact subset
of $\C^k \times \C^n$. As $\phi_1$ is a proper holomorphic
embedding, Lemma \ref{FGRlemma2copy} gives that $\hat K_2$ is
contained in $ (2\overline \B_k \times 2\overline \B_n)\cup
\phi_1( 1\overline \B_k \times \overline X )$, in particular the
points $(w,a_j^2)$ will not intersect the set $\hat K_2$.

\end{enumerate}

By our dimension assumptions ($k<n-1$) the parametrized points $\{
(w,a^2_j),(w,x^2_j)\}$ are simultaneously standardizable
(Corollary \ref{standp}). We will now use Lemma \ref{FLYTTLEMMAT}
$k(2)$ times to find an automorphism $\alpha_2'\in
\mathrm{Aut}_{hol}^k(\C^n)$ holomorphically depending on $w\in \C^k$
such that $4_2,5_2,6_2$ and $7_2$ are satisfied with
$\epsilon_2/2$ instead of $\epsilon_2$, 
(as we will combine $\alpha'$ with another automorphism $\alpha''$).

Using Lemma \ref{AKlemma} we find an automorphism $\alpha_2''\in
\mathrm{Aut}_{hol}^k(\C^n)$ not moving any point of $\alpha_2'(K_2)$
more than $\epsilon_2/2$, matches the identity up to order $l$ at
the points $(w,\xi_i(w))$ for every $w\in \C^k$ and fixing the
points $(w,a_j^2)$ $j=1,\ldots ,k(2)$ for every $w\in \C^k$. In
addition, for $w\in 2\overline \B_k$,  the submanifold
$\alpha_2''\circ \alpha_2' \circ \phi_1(w,\tilde X)$ of $\C^n$ do
not osculate of order $l$ in any of the points $P_2(\phi_1 (w,x))$
with $x\in (X_{R_1}\cap \tilde X)\setminus U_1$. The composition
$\alpha_2''\circ \alpha_2'$ will then satisfy $4_2,5_2,6_2,7_2$
and $8_2$. Finally choose $R_2$ so large that $9_2$ and $10_2$ are
satisfied.

The induction for step $s$, $s>2$, goes exactly as step 2.

At all steps $s$, we have to make sure that the property not to
osculate of order $l$ for $\phi_{s-1}\circ \phi_0(\tilde X)$ is
preserved in every point of the image of $(X_{R_{s-2}}\cap \tilde
X)\setminus U_{s-2}$. We therefore have to choose $\epsilon_s \leq
\epsilon_{s-1}$ so small that every perturbation of
$\phi_{s-1}\circ \phi_0:\tilde X\hookrightarrow \C^n$ less than
$3\epsilon_s$ on the compact $(X_{R_{s-2}}\cap \tilde X)\setminus
U_{s-2}$ do not destroy that property.

Because of $5_{\mu}$ and the fact that $\epsilon_{\mu}<\frac 1{\mu}$ Proposition 4.1 and 4.2 from \cite{F}  gives that $A=\lim_{\mu \to \infty}A_{\mu}(w,z)=(w,\lim_{\mu \to \infty}\tilde A_{\mu}(w,z))$ exists uniformly on compacts on $\Omega =\cup_{\mu =1}^{\infty}A_{\mu}^{-1}(\mu \B_k \times \mu \B_n)$ and defines a biholomorphic mapping from $\Omega$ onto $\C^k\times \C^n$. By $4_{\mu }$ the set $\C^k \times X$ is contained in $\Omega$.
Since $A_{\mu}\in \mathrm{Aut}_{hol}^k(\C^n)$ for each fixed $w_0\in \C^k$ the map $A_{w_0}:\Omega_{w_0}\to \C^n$, given by $(w,z)\mapsto \tilde A(w,z)$, is a biholomorphic mapping from $\Omega_{w_0}=\{ z\in \C^n: (w_0,z)\in \Omega \}$ onto $\C^n$, and $\Omega_{w_0}$ contains $X$.
Therefore for all $w_0\in \C^k$ the map $\phi_{w_0}$ defined by
$x\mapsto \phi (w_0,x)$ is a proper holomorphic embedding  of $X$.

We will now confirm that the constructed embedding satisfies the
theorem. Property $3_{\mu}$ and $6_{\mu}$ gives for every $n\in
\N$ that $\phi (w,x_j^n)=(w,a^n_j)$ for every $w\in \C^k$ and
$j=1,2,\ldots ,k(\mu )$.

Define $\epsilon =\sum_{i=2}^{\infty}\epsilon_i$, condition
$1_{\mu}$ gives that $\epsilon <\frac 12$.

Now suppose that there is a non-degenerate holomorphic mapping
$F:\C^n\to \C^n$ and that there are two values $w_1\neq w_2\in
\C^k$ of the parameter space such that $F^{-1}(\C^n\setminus
\phi_{w_2}(X)) = \C^n\setminus \phi_{w_1}(X)$ and that
$\phi^{-1}_{w_2}\circ F\circ \phi_{w_1}=t$ for some $t\in T$, i.e.
some element of the family of automorphisms of $X$. In particular
this will hold if for $w_1\neq w_2$ the embeddings $\phi_{w_1}$ and
$\phi_{w_2}$ are equivalent.

By moving the origin by an arbitrary small translation, we can
assume that $JF(0)\neq 0$. Let $\beta
=\prod_{i=2}^{\infty}(1-\frac{\epsilon_i}{2})>0$ and let $\nu_0$
be a number so large that for every $\nu \geq \nu_0$ we have that
$t\in T_{\nu}$, $JF(0)>\frac 1{(\nu \beta )^n}$, $F(0)\in
\frac{\nu +2}{2}\B$ and $w_1,w_2\in \nu \B_k$.

For a given $\nu \geq \nu_0$ choose a natural number $k$ such that
$F(\nu \beta \B_n )\subset (k+2)\B_n$ and $k>\nu +2$. Define
$F_j(z):\B_n \to (k+2)\B_n$ by $F_j(z)=F(z\cdot \nu
\prod_{l=j+1}^k(1-\frac{\epsilon_l}{2}))$ for $j=1,2,\ldots ,k$.
For some fix $j\in \{ \nu +1,\ldots ,k\}$ we have that $F_j(\B_n
)=F(\nu \prod_{l=j+1}^k(1-\frac{\epsilon_l}{2})\B_n )\subset F(\nu
\B_n )$, which by Lemma \ref{LEMMAB} implies that $F_j(\B_n )$ do not
contain any point $a^j_l\in \partial (j+1)\B_n$, $l=1,2,\ldots
,k(j)$ and $j\geq \nu +1$.

In addition we have
$$| JF_j(0)|=|\nu^n\prod_{l=j+1}^k(1-\frac{\epsilon_l}{2})^nJF(0)|>\nu^n\beta^nJF(0)>1$$
and $F_j(0)=F(0)\in \frac{j+1}2\B_n$. Property $2_k$ now gives that
$F_k((1-\frac{\epsilon_k}{2})\B_n )=F_{k-1}(\B_n )\subset (k+1)\B_n$.
Induction from $k$ down to $\nu +1$ gives that $F_{j-1}(\B_n
)\subset (j+1)\B_n$ so, for $j=\nu +1$ we have $F_{\nu}(\B_n )\subset
(\nu +2) \B_n$ and therefore $F(\beta \nu \B_n )\subset (\nu +2)\B_n$
for every $\nu \geq \nu_0$. This growth condition implies that $F$
is an affine mapping, and the fact that $F$ is non-degenerate
means that $F$ is an affine automorphism.

Remember that $(\varphi_1)_w (\tilde X)$ osculates of order $l$ at $\xi_i(w)$, for $w\in \C^k$, and this is preserved by $7_{\mu}$ over the induction.
Hence, we see that $\varphi_w (\tilde X)$ osculates of order $l$
at $\varphi (x_i)=\xi_i(w)$, $i=1,2,\ldots ,m $. From property
$8_{\mu}$ it follows that $\varphi_w (\tilde X)$ do not osculate
of order $l$ in any other point. Due to this, since the affine
automorphism $F$ maps $\varphi_{w_1}(\tilde X)$ to
$\varphi_{w_2}(\tilde X)$, it also maps the set  $\{ \xi_i(w_1)
\}_{i=1}^m$ to the set $\{ \xi_i(w_2) \}_{i=1}^m$. By the choice
of $\{ \xi_i(w)\}$ there is no such automorphism, see Proposition
\ref{pointprop}.
\end{proof}

\begin{lemma}\label{LEMMAA}
In the notation of the proof of Theorem \ref{main} holds:
\begin{multline*}
\varphi^{-1}((\nu -1)\overline \B_k \times (\nu -1)\B_n)\subset P_2((A_{\nu} \circ \phi_0)^{-1}((\nu -1)\overline \B_k \times \nu \B_n))=\\
=P_2(\phi_{\nu}^{-1}((\nu -1)\overline \B_k \times \nu \B_n))
\end{multline*}
\end{lemma}

\begin{proof}
Let $w\in (\nu -1)\overline \B_k$ be a fix point and consider
some $x\in \varphi_w^{-1}((\nu -1)\overline \B_n)$, where
$\varphi_w^{-1}$ is the restriction of $\varphi^{-1}$ to $\{ w\}
\times  \C^n$. This implies that $\varphi_w(x)\in (\nu
-1)\overline \B_n$. Now choose $k_0>\nu$ and $0<\delta <
1-2\epsilon$ such that
\begin{equation}\label{A*}
\phi_{k,w}(x)\in (\nu -1+\delta )\overline \B_n \textrm{ for every }
k\geq k_0 .
\end{equation}
Property $5_{\nu +1}$ gives that $\| \alpha_{\nu +1}(w,z)-(w,z)\|
\leq \epsilon_{\nu +1}$ for $(w,z)\in (\nu +1)\overline (\B_k
\times \overline \B_n )$ and Rouchés theorem that $\alpha_{\nu
+1,w}(\nu \B_n)\supset (\nu -2\epsilon_{\nu})\B_n$, see Remark
\ref{roch}, or in other words $\nu \B_n \supset (\alpha_{\nu
+1,w})^{-1}((\nu -2\epsilon_{\nu})\B_n )$ so
\begin{multline*}
(\phi_{\nu})_w^{-1}(\nu \B_n )=(A_{\nu}\circ \phi_0)_w^{-1}(\nu \B_n )\supset (A_{\nu +1}\circ \phi_0)_w^{-1}((\nu -2\epsilon_{\nu})\B_n )=\\
=(\phi_{\nu +1})^{-1}_w((\nu -2\epsilon_{\nu})\B_n ).
\end{multline*}

Induction using $5_{\nu +2},\ldots ,5_k$ gives
\begin{multline}\label{B*}
(A_{\nu}\circ \phi_0)_w^{-1}(\nu \B_n )\supset (A_k\circ \phi_0)_w^{-1}((\nu -2\sum_{l=\nu}^{k-1}\epsilon_l)\B_n )\supset \\
\supset (A_k\circ \phi_0)_w^{-1}((\nu -2\epsilon )\B_n ) .
\end{multline}
By our choice of $\delta$ we have $\nu -2\epsilon >\nu -1+\delta$,
so \eqref{A*} and \eqref{B*} implies that $\phi_{\mu}^w(x)\in \nu
\B_n$.
\end{proof}

\begin{notation}\label{roch}
To see that $\alpha_{\nu +1,w}(\nu \B_n )\supset (\nu
-2\epsilon_{\nu})\B_n$ holds, we consider the following situation,
$\| \alpha_q(w,z)-(w,z)\| \leq \epsilon_q$ for $(w,z)\in
q\overline \B_k \times q\overline \B_n$.
\begin{multline*}
\| \alpha_q(w,z)-(w,z)\| =\| \alpha_q(w,z)-(w,p)-((w,z)-(w,p))\| \leq \\
\epsilon_q <\| (w,z)-(w,p)\| \textrm{ for } z\in q\overline \B_k
\times \partial q\overline \B_n \textrm{ and } p\in
(q-2\epsilon_q)\overline \B_n .
\end{multline*}
Since $(w,z)-(w,p)$ has a root and consequently, by Rouché,
$\alpha_q(w,z)-(w,p)$ will too. So for every $p\in
(q-2\epsilon_q)\overline \B_n$ we always have a solution to the
equation $\alpha_q(w,z)=(w,p)$ for some $z\in q\overline \B_n$,
therefore we draw the conclusion that $\alpha_{\nu +1,w}(\nu \B_n
)\supset (\nu -2\epsilon_{\nu})\B_n$.
\end{notation}

\begin{lemma}\label{LEMMAB}
In the notation of the proof of Theorem \ref{main} holds: For
every $j\geq \nu+1$ we have that $ F(\nu \B_n)\cap
\cup_{l=1}^{k(j)}\{ a^j_l\} =\emptyset$.
\end{lemma}

\begin{proof}
Suppose, to reach a contradiction, that there exist $z\in \nu \B_n$
such that $F(z)=a^j_l$ for some $j\geq   \nu +1$ and some $l$
between $1$ and $k(j)$. Since $F^{-1}(\C^n\setminus
\varphi_{w_2}(X))= \C^n\setminus \varphi_{w_1}(X)$, we have that
$z\in \varphi_{w_1}(X)$.

Let $x=\varphi^{-1}_{w_1}(z)\in \varphi^{-1}_{w_1}(\nu \B_n )$,
which gives $F\circ
\varphi_{w_1}(x)=a^j_l=\varphi_{w_2}(x^j_l(w_2))$. Thus
$t(x)=\varphi^{-1}_{w_2}\circ F\circ
\varphi_{w_1}(x)=x^j_l(w_2)$. Using Lemma \ref{LEMMAA} we conclude
that $x\in P_2(\phi^{-1}_{\nu +1,w_1}((\nu +1)\overline \B_n))$
and with $t\in T_{\nu}$ and $w_1\in \nu \B_k$ it follows that
$$\rho (t(x))\leq \max_{\phi_{\nu +1} ^{-1}(\nu +1)\overline \B_n ,t\in T_{\nu},w\in \nu \B_k}\rho(t(y)). $$
Since $j\geq \nu +1$ we have
$$\rho (t(x))\leq \max_{\phi_{j,w}^{-1}j\overline \B_n ,t\in T_{j-1},w\in (j-1) \B_k}\rho(t(y)), $$
(for $j>\nu +1$ we have that $\phi_j$ maps the inverse image
$\phi^{-1}_{j-1}((j-1)\B_n )$ into $(1+ \epsilon_{j})(j-1)\B_n \subset
j\B_n $). However condition $3_j$ gives
$$\rho (x^j_l(w))>\max_{y\in \phi^{-1}_{j-1,w}(j\overline \B_k \times j\overline \B_n), t\in T_j}\rho (t(y)).$$
Therefore $\rho(x_l^j(w_2))>\rho (t(x))$, which contradicts
$t(x)=x_l^j(w_2)$.
\end{proof}

\section{Eisenman hyperbolicity of the embeddings} \label{Eisenman}
Let $M$ be a complex manifold of dimension $n$. We denote the
holomorphic tangent bundle of $M$ by $TM$ and the holomorphic
tangent space at $p\in M$ by $T_pM$. The $k$-th exterior power of
$T_pM$ and $TM$ will be denoted by $\bigwedge^kT_pM$ and
$\bigwedge^kTM$. Let also $D^k_pM$ and $D^kM$ denote the set of
decomposable elements in $\bigwedge^kT_pM$ and $\bigwedge^kTM$.

Recall
that the Eisenman $n$-norm for a $u\in D^n_pM$ is defined as \cite{E}, \cite{GW}
$$E_n^M(p,u)=
\inf \{ \| v\|^2:v\in D_0^n\B_n, \exists F\in\hol (\B_n
,M),F(0)=p, F_*v=u\} .$$

A complex manifold is called $n$-Eisenman hyperbolic if
$E_n^M(p,u) >0$ for all $p\in M$ and all non-zero $u\in D^n_pM$. Compare with \cite{BF}.

We use the notation from the proof of Theorem \ref{main}.
\begin{theorem}[Addition to Theorem \ref{main}]\label{mainadd}
For all $w\in \C^k$ the complement $\C^n \setminus \varphi_w(X)$
of the embedding $\varphi_w(X)$ is Eisenman $n$-hyperbolic.
\end{theorem}

\begin{proof}
Suppose there exists a point $p\in \C^n\setminus \varphi_w (X)=M$
such that $ E_n^M(p,u)=0$ for the (unique up to a constant)
non-zero $u\in D^n_pM$. This means that
\begin{equation}\label{eishyp}
\inf_f \frac{1}{| Jf(0)|^2}\paren{\frac i2}^ndz\wedge d\bar z=0
\end{equation}
for some point $p$ where $f\in \hol (\B_n ,M)$ such that $f(0)=p$
and $f_*(T_0\B_n)=v$.

Let $\nu \in \N$ be a fixed number such that $p\in \frac{ \nu +
2}{2}\B_n$. By \eqref{eishyp} there is $F:\B_n\to M$ such that
$F(0)=p$ and $¦JF(0)¦$ is arbitrary large, for example
\begin{equation}\label{eishyp2}
¦JF(0)¦>\max (\frac{1}{\beta^n}
,(\nu +2)^n\beta^n)
\end{equation}
There is an $\alpha \in \C$, $0<¦\alpha ¦<1$ so that $¦JF(0)>\frac
1{\alpha^n\beta^n}$. Since $F(\alpha \overline{\B}_n)$ is compact
we find $k\in \N$ such that $F(\alpha \B_n)\subset (k+2)\B_n$.

Define $F_j(z)=F(\alpha \prod_{l=j+1}^k(1-\frac{\epsilon_j}2)z)$.
It holds: 

$¦JF_j(0)¦\geq \alpha^n\beta^n¦JF(0)¦>1$ for every $j$,
also for $\nu +1 \leq j \leq k$ we have $F_j(0)=p\in \frac{ \nu +
2}{2}\B_n$ and $F_j(\B_n)$ obviously does not meet the points
$a^j_l$, $l=1,2,\ldots ,k(j)$, (for $j$ large enough as in the
proof of the main theorem).

We conclude inductively by property $2_j$ $F_{j-1}(\B_n)\subset
(j+1)\B_n$ for $\nu +1 \leq  j\leq k$. This means in particular
$F_{\nu}(\B_n)\subset (\nu +2)\B_n$ which implies
$$¦JF_{\nu}(0)¦\leq (\nu +2)^n$$
and therefore
$$¦JF(0)¦\leq (\nu +2)^n\alpha^n\beta^n .$$
This contradicts \eqref{eishyp2}, thus $\C^n\setminus \varphi_w
(X)$ is Eisenman $n$-hyperbolic.
\end{proof}

That Eisenman hyperbolic manifold have a cancellation property was
used in \cite{Z}, Theorem 1.10., (for a simple proof see for example \cite{Bo})

\begin{proposition}\label{cancel}
Let $Y$ and $Z$ be $n$-Eisenman hyperbolic manifolds. Then any
biholomorphic map $\Psi =(\psi_1,\psi_2 ):Y\times \C^l \to Z\times
\C^l$ is of the form $\Psi (y,z)=(\psi_1(y),\psi_2(y,z))$, where
$\psi_1:Y\to Z$ is biholomorphic.
\end{proposition}

With Proposition \ref{cancel} and Theorem \ref{mainadd} we get

\begin{theorem}\label{cross}
Let $X$ be a complex space, which can be embedded in $\C^n$ and
such that the group of holomorphic automorphisms
$\aut_{\rm{hol}}(X)$ is a Lie group. Then there exist, for
$k=n-1-\dim X$,  a family of holomorphic embeddings of $X\times
\C^l$ into $\C^n\times \C^l$ parameterized by $\C^k$, such that
for different parameters $w_1\neq w_2\in \C^k$ the embeddings
$\psi_{w_1},\psi_{w_2}:X\times \C^l \hookrightarrow \C^{n+l}$ are
non-equivalent (up to automorphisms).
\end{theorem}

\begin{proof}
Take $\varphi$ from Theorem \ref{main} and consider
$$\Psi :\C^k\times X\times \C^l \to \C^k\times \C^n\times \C^l$$
defined by
$$(w,x,y)\mapsto (w,\varphi (w,x),y)=(w,\psi(w,x,y)) .$$
Assume that the embeddings $\psi_{w_1}$ and $\psi_{w_2}$, where
$w_1\neq w_2$ are equivalent. This means that there exists an
automorphism $\alpha \in \aut_{\rm{hol}}(\C^{n+l})$ such that
$\alpha(\varphi_{w_1} (X)\times \C^l)=\varphi_{w_2} (X)\times
\C^l$ and therefore the same for the complements
$\alpha((\C^n\setminus \varphi_{w_1} (X))\times
\C^l)=(\C^n\setminus \varphi_{w_2} (X))\times \C^l$.

Now by Proposition \ref{cancel} there exists $\alpha_1 \in
\aut_{\rm{hol}}(\C^n)$ such that $\alpha_1(\C^n\setminus
\varphi_{w_1} (X))=\C^n\setminus \varphi_{w_2} (X)$. Thus
$\alpha_1(\varphi_{w_1} (X))=\varphi_{w_2} (X)$, which contradicts
the choice of $\varphi$.
\end{proof}

A special case which is worth to state separately is $X=\C$.

\begin{corollary} \label{main corollary}
There exist, for $k=n-l-1$,  a family of holomorphic embeddings of
$\C^l$ into $\C^n$ parameterized by $\C^k$, such that for
different parameters $w_1\neq w_2\in \C^k$ the embeddings
$\psi_{w_1},\psi_{w_2}:\C^l \hookrightarrow \C^{n}$ are
non-equivalent.
\end{corollary}

We end this section with a little trick showing that one can even have families of pairwise non-equivalent embeddings containing the standard 
embedding as a member of the family. Note that the embeddings constructed in the proof of Theorem
\ref{main} are not containing the standard embedding since the complement of each embedding is
$n$-Eisenman hyperbolic. 

\begin{proposition} \label{standard}
For each $0<l<n-1$ there is a holomorphic   family of holomorphic embeddings of
$\C^l$ into $\C^n$ parameterized by $\C$, such that for
different parameters $w_1\neq w_2\in \C$ the embeddings
$\psi_{w_1},\psi_{w_2}:\C^l \hookrightarrow \C^{n}$ are
non-equivalent. Moreover for the embedding $\psi_0$ is equivalent to the standard embedding.
\end{proposition}

\begin{proof}
Take a family $\Phi : \C \times \C \to \C \times \C^{n-l+1}$  $\Psi_0 (w, z) = (w, \phi (w,z))$ as constructed in Theorem \ref{main} and cross
it with $\C^{l-1}$ as in Theorem \ref{cross}  to get a family $\Psi_1 : \C \times \C^l \to \C \times \C^n$
$\Psi_1 (w, x) = (w, \phi_1 (w,x))$. By using a translation we can assume that $\phi_1 (w, 0) = 0 \ \forall w\in \C$.

Now define the family $\Psi : \C \times \C^l \to \C \times \C^n$ by
$$ \Psi (w, x) = (w, {1 \over w} \phi_1 (w, w x)) =: (w, \psi (w,x))$$ for $w\ne 0$ and by its obvious limit $ x \mapsto \phi_1^\prime  (0,0) x$ for $w=0$. Thus for $w=0$ we have the standard embedding in the family. All other members $\psi_w$ of
the family are by definition equivalent to the embedding $\phi_w$ and therefore pairwise non-equivalent. No member in the family except $\psi_0$ is equivalent to the standard embedding since
otherwise there would exist a holomorphic map of rank $n-l+1$ into the complement of $\phi_{w} (\C)$ which contradicts  the Eisenman $n-l+1$-hyperbolicity (Theorem \ref{mainadd}).
\end{proof}

\section{Families of holomorphic $\C^*$-actions on affine space}\label{wirkung}

In this section we employ the method from \cite{DK1} and \cite{DK} to construct (non-linearizable) $\C^*$-actions on
affine spaces out of embeddings $\C^l\hookrightarrow \C^n$. We will not give all proofs in detail. The
important point we want to check here is that if the embeddings are holomorphically parametrized, then
the resulting $\C^*$-actions depend holomorphically on the parameter.

Let's go through the method:

For an embedding $\varphi: \C^l \to \C^n$ take generators of the ideal $I_{\varphi (\C^l)} < \hol (\C^n)$ of the image manifold, say 
$f_1, \ldots, f_N \in \hol ({\C^n)}$ (in this case $N= n-l$ would be sufficient, since $\C^l$ is always a
complete intersection in $\C^n$ by results of Forster and Ramspott \cite{FoRa}, but this is not important for the construction) and consider the manifold

\begin{multline*}
M:= \{ (z_1, \ldots, z_n, u_1,\ldots u_N, v) \in \C^{n+N+1} : \\
f_i (z_1, \ldots, z_n) = u_i \ v      \quad \forall \ i=1, \ldots, N \}
\end{multline*}

which in \cite{DK1} is called Rees space. This notion was introduced there by the authors since they were not aware of the fact that this is a well-known construction, called affine modification, going back to Oscar Zariski. Geometrically the manifold $M$ results from $\C^{n+1}_{z, v}$ by blowing up along the center $\mathcal{C} = \varphi (\C^l) \times 0_v$ and deleting the proper transform of the divisor
$\mathcal{D} = \{ v = 0\}$. Since our center is not algebraic but analytic,  the process usually is called
pseudo-affine modification.

Lets denote the constructed manifold $M$ by $Mod (\C^{n+1}, \mathcal{D}, \mathcal{C})  = 
Mod( \C^{n+1}_{z, v}, \{v = 0\}, \varphi (\C^l)\times \{v=0\})$. It's clear from the geometric description that
the resulting manifold does not depend on the choice of generators for the ideal $I_\mathcal{C}$ of the center. 

The important fact about the above modifications is that 

\noindent
$Mod( \C^{n+1}_{z, v}, \{v = 0\}, \varphi (\C^l)\times \{v=0\}) \times \C^l$ is biholomorphic to $\C^{n+l+1}$
$\cong Mod (\C^{n+l+1}_{z, u, v}, \{ v=0\}, \varphi(\C^l) \times 0_u \times 0_v)$. The later biholomorphism
comes from the fact that there is an automorphism  of $\C^{n+l+1}$ leaving the divisor $\{ v= 0\}$ invariant
and straightening the center $\varphi(\C^l) \times 0_v$ inside the divisor (see Lemma 2.5. in \cite{DK1}).

Lets check that this important fact depends holomorphically on the parameter.

\begin{lemma}\label{straight}
Let $\Phi_1 : \C^k \times X \hookrightarrow \C^k \times \C^n$, $\Phi_1 (w, x) = (w, \varphi_1 (w, x))$ and
$\Phi_2 : \C^k \times X \hookrightarrow \C^k \times \C^m$, $\Phi_2 (w, x) = (w, \varphi_2 (w, x))$ be two holomorphic families of proper holomorphic embeddings of a complex space $X$ into $\C^n$ resp. $\C^m$ parametrized by $\C^k$. Then there is an automorphism $\alpha$ of $\C^{n+m}$ parametrized by $\C^k$, i.e. $\alpha \in \aut_{hol} (\C^k_w \times \C^{n+m}_z)$ with $\alpha (w, z) = (w, \tilde \alpha (w, z))$, such that $\alpha \circ (\Phi_1 \times 0_m) = 0_n \times \Phi_2$.
\end{lemma}

\begin{proof}
By an application of Theorem B the holomorphic map $\varphi_1 : \C^k \times X $ to $ \C^n$ extends to a holomorphic map $\mu_1$ from $\C^k \times \C^m \supset \Phi_2 (\C^k \times X)$ to $\C^n$ (so $\mu_1\circ \varphi_2 = \varphi_1$). Likewise there is a holomorphic map $\mu_2 : \C^k\times \C^n \to \C^m$
with $\mu_2 \circ \varphi_1 = \varphi_2$. Define the parametrized automorphisms $\alpha_1, \alpha_2$
of $\C^k\times \C^n\times \C^m$ by $\alpha_1 (w, z, y) = (w, z, y+ \mu_2 (w, z))$  and 
$\alpha_2 (w, z, y) = (w, z + \mu_1 (w, y), y)$. Now $\alpha = \alpha_2^{-1} \circ \alpha_1$ is the desired 
automorphism.
\end{proof}

\begin{lemma} 
\label{fam}
Let $\Phi : \C^k \times \C^l \hookrightarrow \C^k \times \C^n$ $\Phi (w, \theta) = (w, \varphi (w, \theta))$ be a holomorphic family of proper holomorphic embeddings of $\C^l$ into $\C^n$ parametrized by $\C^k$.

Then $Mod( \C^{k+n+1}_{w, z, v}, \{ v=0\}, \Phi (\C^k\times \C^l )\times \{ v=0 \}) \times \C^l \cong \C^{k+n+l+1}$. Moreover there is a  biholomorphism such that the restriction to each fixed parameter  $w\in \C^k$ is a biholomorphism from $ Mod( \C^{n+1}_{ z, v}, \{v=0\}, \Phi(\{w\}\times \C^l)\times \{v=0\}) \times \C^l \cong \C^{n+l+1}$.
\end{lemma} 

\begin{proof}
Apply Lemma \ref{straight} to the families $\Phi_1 = \Phi$ and $\Phi_2$ the trivial family
$\Phi_2 : \C^k \times \C^l \hookrightarrow \C^k\times \C^l$ $\Phi_2 (w, \theta) = (w, \theta)$. Let $\alpha
\in \aut_{hol} (\C^k \times \C^n \times \C^l)$ be the resulting parametrized automorphism which we extend to $\C^{k+n+l+1}$ by letting it act trivial on the last coordinate $v$. Then by definition
$Mod( \C^{k+n+1}_{w, z, v}, \{ v=0\}, \Phi (\C^k\times \C^l )\times \{ v=0 \}) \times \C^l = 
Mod( \C^{k+n+l+1}_{w, z,\theta,  v}, \{ v=0\}, \Phi (\C^k\times \C^l )\times \{ v=0 \} \times 0_l)$ and applying (the extended) $\alpha$ we get that the later is biholomorphic to
$Mod( \C^{k+n+l+1}_{w, z,\theta,  v}, \{ v=0\}, \C^k_w \times 0_n \times \C^l_\theta \times \{ v=0 \} )$. The last manifold is obviously biholomorphic to $\C^{k+n+l+1}$ since blowing up along a straight
center and deleting the proper transform of a straight divisor does not change the affine space.
The above constructed biholomorphism restricts to each fixed parameter as desired since $\alpha$ is
a parametrized automorphism. This can be also seen by writing down concrete formulas for the modifications using generators $f_1 (w,z), \ldots, f_N(w, z)$ of the ideal $I_{\phi(\C^k\times\C^l)}$ in
$\hol (\C^{k+n})$  and remarking that for each fixed $w \in \C^k$ the functions $f_1 (w,\cdot), \ldots, f_N(w, \cdot )$ generate the ideal $I_{\Phi_w (\C^l)}$.
\end{proof}

Now we describe the group actions:

\noindent
Let  $f_1 (w,z), \ldots, f_N(w, z)$ be generators of the ideal $I_{\phi(\C^k\times\C^l)}$ in
$\hol (\C^{k+n})$ and consider $Mod( \C^{k+n+1}_{w, z, v}, \{ v=0\}, \Phi (\C^k\times \C^l )\times \{ v=0 \}) \times \C^l \cong \C^{k+n+l+1}$ as the affine manifold given by equations:

\begin{multline*} 
\{ (w, z, v, u) \in \C^k\times \C^n\times \C\times \C^N :
f_i (w, z) = u_i \ v      \quad \forall \ i=1, \ldots, N \} \times \C^l_x
\end{multline*} 

On it we consider the action of $\C^*_\nu$ given by the restriction of the following linear action on the ambient space:

\begin{multline}
\C^* \times \C^k \times \C^n \times \C \times \C^N \times \C^l \to \C^k \times \C^n \times \C \times \C^N \times \C^l \\(\nu, (w, z, v, u, x)) \mapsto (w, z, \nu^2 v, \nu^{-2} u_1, \ldots ,\nu^{-2} u_N, \nu x_1, \ldots, \nu x_l)
\end{multline}

This gives by Lemma \ref{fam} a holomorphic family of $\C^*$-actions on $\C^{n+l+1}$ parametrized by $\C^k$, i.e., an action
$\C^* \times \C^k \times \C^{n+l+1} \to \C^k \times \C^{n+l+1}$ of the form $(\nu (w, z)) \mapsto (w, \nu (w, z))$. Calculating (as in \cite{DK}) the Luna-stratification of the categorical quotient $\C^{n+l+1}/  \hspace{-3 pt} / \C^*$ for the $\C^*$-action  for fixed $w$, in particular the inclusion of the fixed point stratum in the $\Z/2\Z$-isotropy stratum one sees that this inclusion is biholomorphic to $\Phi_w (\C^l) \subset \C^n$. Thus if
for different parameters $w_1 \ne w_2$ there were an equivariant automorphism $\alpha \in \aut_{hol} (\C^n)$ the induced isomorphism of the categorical quotients would map the Luna-stratifications onto each other. Therefore the restriction of that induced isomorphism to the $\Z/2\Z$-isotropy stratum would
give an automorphism $\beta$ of $\C^n$ with $\beta (\Phi_{w_1} (\C^l)) = \Phi_{w_2} (\C^l)$. This shows
that pairwise non-equivalent embeddings lead to non-equivalent $\C^*$-actions.
Combining this with Theorem \ref{main}  (embeddings of $\C$ into $\C^n$ parametrized by $\C^{n-2}$ for $n\ge 3$) we have proved Theorem \ref{action1} from the introduction. In the same way Theorem \ref{action2} from the introduction follows from Proposition \ref{standard}. It's an easy exercise that a straight embedding leads to a linear action.

\section{Concluding remarks}\label{concluding}

Carefully examining the proof of Theorem \ref{main} and the proofs of the technical  results from section \ref{technik} one sees that there is no place where we use the fact that the parameter space is affine space $\C^k$. What we use of the parameter space is a graduation (in the proof of the parametrized Anders\'en-Lempert-theorem) so say an affine algebraic 
variety would do the job. Most important is the dimension condition $\dim X + \dim (\rm {parameter space}) < n$ (here dimension is always dimension of the smooth part). So in fact we construct  families
parametrized by any space of the right dimension. The authors wonder whether there is any nice structure on the set of all equivalence classes of proper holomorphic embeddings  say of $\C^l$ into $\C^n$ and how "big" is this set?

Our construction of embeddings used two techniques, the growth restrictions which worked well for embedding manifolds with a "small" automorphism group, namely a Lie group, and the Eisenman hyperbolicity for crossing the situation with affine space. Combining this we got families of embeddings of affine spaces.

What about the number of equivalence classes of proper holomorphic embeddings of other manifolds with infinite-dimensional automorphism groups, e.g. manifolds with the density property, into affine spaces? A concrete question in this direction would be:

 {\em How many
embeddings of a Danielevski surface $f(X) = u v$ into affine spaces do there exist?}

 It's known that there exist  at least two algebraic embeddings of the Danielevski surface $p(y) = u^n v$ (degree of $p$ is at least 2) into $\C^3$ which are algebraically non-equivalent, i.e. there is no algebraic automorphism of $\C^3$ mapping one image onto the other \cite{FrMo}. In the same paper Freudenburg and Moser
  show that the constructed embeddings are holomorphically isomorphic using the linearization results of Heinzner and the first author \cite{HK}. On the other hand there is a non-standard holomorphic
embedding of the Danielevski surface into $\C^3$, which follows from the ideas of Rosay and Rudin \cite{RR1}:

\begin{proposition} Any algebraic subvariety $A$ in $\C^n$ ($n\ge2$) admits another holomorphic embedding into $\C^n$ not isomorphic to the inclusion.
\end{proposition}
\begin{proof}
The restriction of a generic projection onto a hyperplane  to $A$ is a proper map. Thus by the results in \cite{RR1}
any discrete sequence of points in $A$ is tame (in fact very tame). Now there is a holomorphic embedding $\varphi$ of $A$
into $\C^n$ (constructed by applying a sequence of automorphisms to the inclusion) such that $\varphi (A)$ contains a non tame set $F$ (details as in \cite{FGR}). The existence of a holomorphic automorphism mapping $A$ onto $\varphi (A)$ contradicts the non tameness of $F$.
\end{proof}

It would be interesting to know under which conditions parametrized points (by any parameter space and in any category, continuous, holomorphic differentiable, algebraic) are simultanuously
standardizable.

\section{APPENDIX: proofs of  technical preparations}\label{prooflemmas}

In this section we give the proofs of the lemmas etc. used in the proof of the main theorem. 

\subsection{A parametrized version of the Andersén-Lempert theorem}

\medskip

Our main technique we use to construct  families of embeddings are (compositions of)
automorphisms of $\C^n$.

The ground-breaking papers of Anders\' en and Lempert (\cite{A},
\cite{AL}) established remarkable properties of the automorphism group
of $\C^n$ $(n \ge 2 )$ which imply, in particular, that any local
holomorphic phase flow on a Runge domain $\Omega$ in $\C^n$ can be
approximated by global holomorphic automorphisms of $\C^n$ (for an
exact statement see Theorem 2.1 in \cite{FR}). We will give here a parametrized version
of the so called Andersén-Lempert-theorem  and in addition we consider the following two geometric structures:
that of vector fields vanishing on the first $N$ standard points in $\C^n$, and that of vector fields vanishing on the first coordinate axis. Since the parametric version is
an easy consequence of the non-parametric version and the fixing of the first $N$-standard points is a special case of Theorem 6 in \cite{KaKu} we just give a small indication of the proof.

%

\begin{theorem}[Andersén-Lempert-theorem with parameter and fixing finitely many points]\label{alwpfc}

Let $\Omega$ be an open set in $\C^k\times \C^n$ ($n\geq 2$) and
let $(w,z_j)=(w,j,0,\ldots ,0)\in \Omega$, $j=1,\ldots
,N$. For every $t\in [ 0,1]$ let $\Phi_t$ be a biholomorphic map
from $\Omega$ into $\C^k\times \C^n$, which is of the form
$$\Phi_t(w,z)=(w,\varphi_t(w,z)), \quad z\in \C^n,w\in \C^k$$
such that $\Phi_t(w,z_j)=(w,z_j)\ \forall w\in \C^k$ (resp. $\Phi_t(w,z_1,0,\ldots, 0)=(w,z_1,0, \ldots, 0)\ \forall z_1 \in \C \ \forall w\in \C^k$) and such that it is of class
$C^2$ in $(t,z,w)\in [0,1]\times \Omega$. Assume that each domain
$\Phi_t(\Omega )$ is Runge in $\C^k\times \C^n$. If $\Phi_0$ can
be approximated on $\Omega $ by holomorphic automorphisms of
$\C^n$ depending on the parameter $w\in \C^k$, fixing $(w,z_j)$ for every $w\in \C^k$ (resp. fixing $(w,z_1,0, \ldots, 0)\  \forall z_1 \in \C \ \forall w\in \C^k$) then for every
$t\in [0,1]$ the map $\Phi_t$ can be approximated on $\Omega$ by
holomorphic automorphisms $\alpha$ of $\C^n$ depending on the parameter
$w\in \C^k$ such that  $\alpha (w,z_j)=(w,z_j) \  \forall w\in \C^k$ (resp. $\alpha (w,z_1,0,\ldots, 0)=(w,z_1,0, \ldots, 0)\ \forall z_1 \in \C \ \forall w\in \C^k$).

\end{theorem}

To indicate the proof we just remark that the above theorem follows by standard techniques from
the following version of the

\bigskip\noindent
{\bf Anders\'en-Lempert-observation:}

\textit{ Every polynomial vector field on $\C^{k+n}$ ($n\geq 2$) of the form 
\begin{multline*}
X=p_1(w_1,\ldots ,w_k,z_1,z_2,\ldots ,z_n)\frac{\partial}{\partial z_1}+ \ldots +\\
+ p_n(w_1,\ldots ,w_k,z_1,z_2,\ldots ,z_n)\frac{\partial}{\partial
z_n}.
\end{multline*}  vanishing at the first $N$ standard points, i.e. with $$p_i ( w,z_j) = 0 \ \forall i = 1, \ldots , n \ \forall j= 1, \ldots , N ,$$
is a
finite Lie combination  of completely integrable polynomial vector fields of the above form  vanishing at the first $N$ standard points. The same holds if we consider polynomial vector fields vanishing on the first coordinate line instead.}

\bigskip
To prove this observation we develop $X$ by powers of $w$

$$ X =  \sum_\alpha w^\alpha X_\alpha $$ and remark that the polynomial vector fields $X_\alpha$ on $\C^n$
vanish at the first $N$ standard points (resp. on the first coordinate line). By Theorem 6 in \cite{KaKu} (the union of the first $N$ standard points is an algebraic subset of $\C^n$ of codimension at least 2) 
(resp. by Theorem 5.1 in \cite{V})
they can be written as a finite Lie combination of globally integrable polynomial fields on $\C^n$ vanishing on the first $N$ standard points (resp.  on the first coordinate line) say $\theta^i_\alpha$ $i= 1, \ldots, N(\alpha)$. The same Lie combination with $\theta^1_\alpha$ replaced by $w^\alpha \theta^1_\alpha$
(which is still globally integrable on $\C^{n+k}$, on each orbit the factor $w^\alpha$ is a constant) yields
$w^\alpha X_\alpha$ as a Lie combination of globally integrable fields. Summing up over the multiindex $\alpha$ we get the desired result.


\subsection{Families of generic finite sets with respect to affine automorphisms}

As already mentioned we will use growth restrictions to prove that the embeddings for different parameters are not equivalent.
The conclusion of these growth conditions  will be the following: If
two different embeddings in our family are equivalent, then their
images can be mapped onto each other only by some affine
automorphism of $\C^n$. Although it is very unlikely that the
images of two different and more or less complicated embeddings
can be mapped onto each other by affine automorphisms, we must be
accurate in excluding this possibility. Here are some technical
preparations to this point.

If we choose $n+1$ points $x_1,x_2,\ldots ,x_{n+1}$ such that the
difference vectors $x_1-x_i$, $i=2,3,\ldots ,n+1$ form a basis of
$\C^n$, i.e. $x_1,x_2,\ldots ,x_{n+1}$ do not all lie on some
affine hyperplane, then with a little linear algebra we reach the
following conclusion:

\begin{lemma}\label{pointlemma}
For each $m\geq n+2$ there exist $m$ points $x_1,x_2,\ldots
,x_m\in \C^n$ with the following property:  No affine automorphism
$\alpha \in \rm{Aff}(\C^n )$ of $\C^n$ can map $n+2$ of them into
the set $\{  x_1,x_2,\ldots ,x_m\}$.
\end{lemma}

\begin{notation}
Given any open subset $\Omega$ of $\C^n$ the points
$x_1,x_2,\ldots ,x_m$ can be chosen to be contained in $\Omega$.
\end{notation}

Let $\delta_{ij}$ denote the diagonal
$$\delta_{ij}=\{ (z_1,z_2,\ldots , z_N)\in (\C^n)^N:z_i=z_j\}$$
and $(\underbrace{\C^n\times \cdots \times \C^n}_N\setminus
\bigcup_{1\leq i<j\leq N}\delta_{ij})/S_N$ is the quotient
(manifold since we have excluded all diagonals) by the action of
the symmetric group $S_N$ in $N$ letters acting by permuting the
entries on $N$-tuples of points in $\C^n$. The corresponding map
is denoted by $\pi$.

\begin{lemma}\label{nnklemma}
Let $n\geq 2$ and $k>0$ be natural numbers. Then there exists some
$N\in \N$ such there is an injective holomorphic map
$$\varphi :\C^k \to \underbrace{\C^n\times \cdots \times \C^n}_N\setminus \bigcup_{1\leq i<j\leq N}\delta_{ij}$$
such that the composition map
$$\pi \circ \varphi :\C^k \to (\underbrace{\C^n\times \cdots \times \C^n}_N\setminus \bigcup_{1\leq i<j\leq N}\delta_{ij})/S_N$$
is injective. Moreover if $\varphi =(\varphi_1,\varphi_2,\ldots
,\varphi_N)$ then $\C^n\setminus \cup^N_{i=1}\varphi_i(\C^k)$
contains some nonempty open subset.
\end{lemma}

\begin{proof}
The desired number $N$ will be any number such that $n\cdot N\geq
k$. Namely we will prove that there is a biholomorphic image
$\Omega$ of $(\C^n)^N$ in $(\C^n)^N$ (a so called Fatou-Bieberbach
domain) which does not intersect any of the diagonals
$\delta_{ij}$ and so that the restriction of the quotient map $\pi
:(\C^n)^N \to ((\C^n)^N)/S_N$ onto $\Omega$ is injective, i.e. if
the point $(z_1,z_2,\ldots ,z_N)$ is in $\Omega$ then for any
permutation $\sigma \in S_N\setminus \textrm{Id}$ the point
$(z_{\sigma (1)},z_{\sigma (2)},\ldots , z_{\sigma (N)})$ is not
contained in $\Omega$.

For this start with countably many pairwise disjoint
Fatou-Bieber\-bach domains $\Omega_1,\Omega_2,\ldots
,\Omega_N,\ldots$ in $\C^n$. Such domains exist, see for example
\cite{EFW}, where countably many pairwise disjoint
Fatou-Bieberbach domains are constructed arising as basins of
attraction  of some automorphism of $\C^n$ having countably many
attracting fixed points. Now take $N$ of them and denote by
$\psi_i$ some biholomorphic maps $\psi_i:\C^n\to \Omega_i\subset
\C^n\quad i=1,2,\ldots ,N$.

The map
$$\psi :(\C^n)^N\to (\C^n)^N, \psi (z_1,z_2,\ldots ,z_N)=(\psi_1(z_1),\psi_2(z_2), \ldots ,\psi_N(z_N))$$
is injective, its image is the Fatou-Bieberbach domain
$\Omega_1\times \Omega_2\times \cdots \times \Omega_N$ in
$(\C^n)^N$, which does not intersect any diagonal since the
$\Omega_i$'s are pairwise disjoint and for the same reason $\Omega
\cap \sigma (\Omega )=\emptyset$ for any permutation $\sigma \in
S_N\setminus    \textrm{Id}$. The complement $\C^n\setminus
\cup_{i=1}^N(\C^k)$ contains the union
$\cup_{i=N+1}^{\infty}\Omega_i$ of all remaining Fatou-Bieberbach
domains, hence a non-empty subset. For $n\cdot N\geq k$ we can
choose some injective holomorphic map $\alpha :\C^k \mapsto
(\C^n)^N$ and put $\varphi :=\psi \circ \alpha$, which is the
desired map.
\end{proof}

Now we are able to prove Proposition \ref{pointprop}.

%

\begin{proof}[Proof of Proposition \ref{pointprop}]
By Lemma \ref{nnklemma} we can find $N$ (for $n\cdot N\geq k$)
(pairwise different) points $\xi_1,\xi_2, \ldots ,\xi_N$ in
$\C^n$ parametrized by $\C^k$ such that for different parameters
$w_1\neq w_2\in \C^k$ the set of points $\{
\xi_1(w_1),\xi_2(w_1),\ldots ,\xi_m(w_1)\}$ and  $\{
\xi_1(w_2),\xi_2(w_2),\ldots ,\xi_m(w_2)\}$ are different. Choose
$M\in \N$ such that $M-N\geq n+2$. By Lemma \ref{pointlemma} we
can find $M$ points $x_1,x_2,\ldots ,x_M$ in $\C^n$ such that no
affine automorphism except the identity can map $n+2$ of them into
the set $\{ x_1,x_2,\ldots ,x_M\}$. Choose the next parametrized
points $\xi_{N+1},\xi_{N+2},\ldots ,\xi_{N+M}:\C^k\to \C^n$ to be
constant
$$\xi_{N+i}(w)=x_i\quad \textrm{ for every } w\in \C^k \quad i=1,2,\ldots ,M.$$
To make the parametrized points $\xi_1,\xi_2,\ldots ,\xi_{M+N}$
pairwise different (for any fixed parameter) we choose the points
$x_1,x_2,\ldots ,x_M$ from some open subset in the complement of
all images $\xi_j(\C^k)\subset \C^n\quad j=1,2,\ldots ,N$.

We claim that $\{ \xi_1(w),\xi_2(w),\ldots ,\xi_{N+M}(w)\}$ satisfy
our condition (so $m=N+M$). Indeed suppose that for two different
parameters $w_1\neq w_2\in \C^k$ there is an affine automorphism
$\alpha \in \mathrm{Aff}(\C^n)$ which map the set of points $\{
\xi_1(w_1),\xi_2(w_1),\ldots ,\xi_{N+M}(w_1)\}$ onto the set of
points   $\{ \xi_1(w_2),\xi_2(w_2),\ldots ,\xi_{N+M}(w_2)\}$.
Since $M-N\geq n+2$ at least $n+2$ of the last $M$ points,
($x_1,x_2,\ldots ,x_M$), are mapped by $\alpha$ into $\{
x_1,x_2,\ldots ,x_M\}$, and note that at most $N$ points among
them can be mapped onto the first $N$ points! By the choice of
$x_1,x_2,\ldots ,x_M$ according to Lemma \ref{pointlemma} this
implies that $\alpha$ is the identity map. But this means that the
identity maps the last $M$ (constant) points onto themselves,
hence the points $\{ \xi_1(w_1),\xi_2(w_1),\ldots ,\xi_N (w_1)\}$
onto the points $\{ \xi_1(w_2),\xi_2(w_2),\ldots ,\xi_N(w_2)\}$,
which is impossible since those sets are different by Lemma
\ref{nnklemma}. Thus no such affine automorphism $\alpha$ exists.

\end{proof}

\subsection{Moving finitely many parametrized points.}
Recall that we needed the notion of simultaneously standardazable points in the proof of the main theorem. Given $N$ parametrized points $\zeta_i:\C^k\to C^n$.  If we can find an automorphism $\psi \in \aut_{hol}^k(\C^n)$ such that
$$\psi (w,\zeta_i(w))=(w,(i,0)) \quad \textrm{for all } i=1,2,\ldots ,N \quad \textrm{and for all } w\in \C^k ,$$  we say that the points are simultaneously standardazable. 

The following theorem is a special case of the Oka-Grauert-Gromov-h principle in complex analysis. Even if our application would fit in the classical context proved by Grauert \cite{Gr1}, \cite{Gr2}, \cite{Gr3} (our transition functions are affine linear, i.e.  contained in a complex Lie group) we formulate it in a more general  (but not too general in order to avoid the discussion of sprays) way. For reference see \cite{FF2} , section 2.3 in \cite{Gro} or Theorem 1.4 in \cite{FP2}.

\begin{theorem}[Oka principle with approximation]\label{oka}
Let $X$ be a Stein manifold and let $Z$ be a locally trivial
bundle such that the fiber $Z_x$ is isomorphic to $\C^n$. If
$s : X\to Z$ is a continuous section which is holomorphic in a
neighborhood of an $\hol (X)$-convex compact subset $K$ then there
exists a holomorphic section $\tilde s: X\to Z$ such that $\tilde
s$ approximates  $s$ uniformly on $K$.
\end{theorem}

\begin{example}
Let $q_1(w),q_2(w),\ldots ,q_n(w):\C^k \to \C$ be holomorphic
functions without common zeros. We want to  find holomorphic functions
$h_1(w),h_2(w),\ldots ,h_n(w) :\C^k \to \C$ such that
\begin{equation*}
h_1(w)q_1(w)+h_2(w)q_2(w)+\ldots +h_n(w)q_n(w)=1
\end{equation*}
for every $w\in \C^k$, i.e. find a point $(h_1(w),h_2(w),\ldots
,h_n(w))$ in the hyperplane given by
\begin{equation}\label{corona}
x_1q_1(w)+x_2q_2(w)+\ldots +x_nq_n(w)=1.
\end{equation}

Since $q_1(w),q_2(w),\ldots ,q_n(w)$ do not have common zeros, we
can find local solutions:  If (e.g.) $q_1(w)\neq 0$ then we can
solve the problem in a neighborhood of $w_0$ by setting
$h_2(w)=\ldots =h_n(w)=1$ and
$$h_1(w)=\frac{1-(h_2(w)q_2(w)+\ldots +h_n(w)q_n(w))}{q_1(w)}.$$

Let $Z$ be the (locally trivial with affine linear transition functions) bundle over
$\C^k$ such that the fiber $Z_w$ is the hyperplane \eqref{corona}
in $\C^n$. Given a holomorphically convex compact set $J\in \C^k$
and a holomorphic section $s$ of the bundle over a neighborhood of $J$ we can, by standard arguments in obstruction theory (all homotopy groups of the fiber vanish), extend it to a
continuous section $s : \C^k\to Z$.
Theorem \ref{oka} gives  a holomorphic section $\tilde s : \C^k\to
Z$ which approximates $s$ uniformly on the compact $J$.
\end{example}

\begin{lemma}\label{te1.1}
Given a holomorphic map $\xi =(\xi_1,\ldots ,\xi_n):\C^k\to \C^n$,
($n\geq 2$), always disjoint from the first $N$ standard points,
$\xi (w)\notin \cup_{i=1}^N\{ (i,0,\ldots ,0)\},$
and such that the functions $\xi_2,\ldots ,\xi_n \in \hol (\C^k)$ have no common zero on $\C^k$.
Then there exist $\alpha \in \rm{Aut}^k_{hol}(\C^n)$ fixing the first $N$ standard points, $\alpha (w,i,0,\ldots ,0)=(w,i,0,\ldots ,0)$ with $\alpha (w,\xi (w))=(w,z_0)$, i.e. $\alpha (w,\xi (w))$ is a constant point $(w,z_0)$ for every $w\in \C^k$.

Moreover given a ball $J=r\overline \B_k \subset \C^k$ and a number  $R>N$ such that for $w\in J$ hold: $| \xi (w)|>R$.
Then for any $\epsilon >0$ the automorphism $\alpha$ can be chosen in such a way that
$$\max_{w\in J, |z|\leq R}|\alpha (w,z)-(w,z)| <\epsilon .$$
\end{lemma}

\begin{proof}
The first step consists in an application of Theorem \ref{alwpfc} with fixing the first coordinate line
to bring the points $\xi (w)$, $w\in J$, arbitrarily nearby to a
constant position. To apply the theorem let $\Omega =
\bigcup_{w\in r'\B_k}\{ w\} \times \{ R'\B_n \cup \epsilon_1\B_n
(\xi (w))\}$, where $\B_n(\xi (w))$ is the  unit ball in $\C^n$ with
center in $\xi (w)$, with $r',R'$ slightly bigger than $r,R$ and
$\epsilon_1$ sufficiently small so that $\epsilon_1\B_n (\xi (w))$
has empty intersection with the first coordinate line for all
$w\in r'\B_k$. Note that $\Omega$ is Runge in $\C^k\times \C^n$.

Approximating the map $\Phi_t:[ 0,1] \times \Omega \to \C^k\times
\C^n$ defined by $\Phi_t(w,z)=(w,z)$ for every $w\in r'\B_k$,
$z\in R'\B_n$ and $\Phi_t(w,z)=(w,\xi ((1-t) w)+z-\xi (w))$ for $w\in r'\B_k ,
z\in \epsilon_1 \B_n (\xi (w))$, gives an $\alpha_1 \in
\rm{Aut}^k_{hol}(\C^n)$:
\begin{multline*}
\alpha_1 (w,(z_1,0,\ldots ,0))=(w,(z_1,0,\ldots ,0)) \ \forall z_1 \in \C \ \forall w\in \C^k\\
|\alpha_1 (w,z)-(w,z)|_{w\in J, z\in R\B_n }<\epsilon \\
|\alpha_1 (w,\xi (w))- (w,\xi (0))| < \epsilon \quad \forall w \in
J
\end{multline*}
where  $\epsilon$ is arbitrarily small.

Remark that the last $n-1$ coordinate functions of $\alpha_1 (w,
\xi (w))$ have no common zero on $\C^k$, since by assumption the
same was true for $\xi (w)$ and the first coordinate line is fixed by $\alpha_1$.

A second application of Theorem \ref{alwpfc} again with fixing the first coordinate line using a $C^2$-path from $\xi (0)$ to the point $(2R,1,0,\ldots ,0)$ not intersecting
the first coordinate line and not intersecting $R'\B_n$ shows that
we in addition can assume
$$|\alpha_1 (w,\xi (w))-(2R,1,0, \ldots ,0)|<\epsilon$$
for $w\in J$.

Denote the coordinate functions of  $\alpha_1 (w,
\xi (w))$ by $(\hut q_1(w),\hut q_2(w),\ldots ,$ $\hut q_n(w))$ and
observe that $\hut q_2(w),\ldots , \hut q_n(w)$ have no common
zeros for  $w\in \C^k$.

Now define functions $\hat h_i\in \hol (J)$ by
$$\hat h_3(w)=\ldots =\hat h_n(w)=1$$
and
$$\hat h_2(w)=\frac{1-\hut q_3(w)\hat h_3(w)-\ldots -\hut q_n(w)\hat h_n(w)}{\hut q_2(w)}$$
for $w\in J$. Note that $\hat h_2(w)\approx 1$ for $w\in J$.

By Theorem \ref{oka} (see example thereafter) we have that for
every $\epsilon>0$ there exist $h_i\in \hol (\C^k)$ with
$$\sum_{i=2}^nh_i(w)\hut q_i(w)=1, \forall w\in \C^k$$
and
$$\| h_i(w)-\hat h_i(w)\| _J<\epsilon ,$$

which implies $h_i(w)\approx 1$ for $w\in J$ for all $i=2, \ldots
,n$. Define an automorphism $\alpha_2\in \rm{Aut}^k_{hol}(\C^n)$
by
\begin{multline*}
(w,z)\mapsto \\
(w, z_1+(2R-\hut q_1(w))[z_2h_2(w)+z_3h_3(w)+\ldots z_nh_n(w)],z_2,z_3,\ldots , z_n).$$
\end{multline*}

It holds:
  $$\alpha_2(w,\hut q(w))=(w,2R,\hut q_2(w), \ldots ,\hut q_n(w))$$
  and
  $$\alpha_2(w, (i,0,\ldots ,0))=(w, (i,0,\ldots ,0))$$

  The next step is to construct an automorphism $\alpha_3\in \rm{Aut}^k_{hol}(\C^n)$ that moves $\alpha_2(w,\hut q(w))$ to $(w,2M,0,\ldots ,0)$. So define a polynomial on $\C$
  $$Q(t)=\left( (t-1)(t-2)\cdots (t-N)\frac 1{(2R-1)(2R-2)\cdots (2R-N)} \right)^H,$$
where $H$ is so large that $|t|<R$ implies $|Q(t)|< \epsilon$, and
define the automorphism $\alpha_3$ by
\begin{multline*}
(w,z)\mapsto \\
(w,z_1,z_2-Q(z_1)\hut q_2(w), z_3-Q(z_1)\hut q_3(w),\ldots
,z_n-Q(z_1)\hut q_n(w)).
\end{multline*}
From this we get
$$\alpha_3 (w,i,0,\ldots ,0)=(w,i,0,\ldots ,0)$$
for $i=1,\ldots N,$
$$\alpha_3\circ \alpha_2(w,\hut q(w))=(w,2R,0,\ldots ,0)$$
and it is easy to check that
$$\max_{w\in J, |z|\leq R} |\alpha_3 \circ \alpha_2\circ \alpha_1(w,z)|$$
is arbitrarily small. The composition $\alpha =\alpha_3 \circ
\alpha_2 \circ \alpha_1$ is our desired automorphism.
\end{proof}

\begin{notation}
The reason that we bring the points $\{ \xi (w), w\in J \}$ first
near to $(2R,1,0,\ldots ,0)$ (instead of $(2R,0,\ldots ,0)$
directly) and make them afterwards constant at the point
$(2R,0,\ldots ,0)$ is the following: Our method could lead to a
big movement of $R\B_n$, as we see in the following example.  
\end{notation}

\begin{example}\label{bigmotion}
Let $K=K_1 \times  \bar \Delta^n$ where $\bar \Delta^n$ is the
closure of the unit polydisc in $\C^n$ and $K_1$ some compact set
in the parameter space $\C^k$. We will consider the following
perturbation of the first $N$ standard points:

We suppose that the first $N-1$ points remains  at their standard
positions,
$$\zeta_i(w)=(i,0,\ldots ,0)\quad i=1,2,\ldots ,N-1$$
and the $N$-th point is moved by some very small amount from the
standard position
$$\zeta_n(w)=(2R-\epsilon ,\epsilon^2,\ldots ,\epsilon^2).$$
According to the proof of Lemma \ref{te1.1} we have to find
holomorphic  functions $h_2(w),h_3(w), \ldots ,h_n(w)$ with
$\sum_{i=2}^Nh_i(w)\epsilon^2=-2R+\epsilon +2R=\epsilon$, i.e.
$\sum_{i=2}^Nh_i(w)=1/\epsilon$. Therefore the automorphism
$\alpha_2\in \mathrm{Aut}^k_{hol}(\C^n)$ defined by
$$(w,z_1,z_2,\ldots ,z_n)\mapsto (w,z_1+z_2h_2(w)+\ldots +z_nh_n(w),z_2,z_3,\ldots ,z_n)$$
moves for instance all points of the form $(0,z_1,1,\ldots ,1)\in
K_1 \times \bar \Delta ^n$ by the vector $(0,1/\epsilon ,0, \ldots
,0)$, i.e. $(0,z_1,1,\ldots ,1)\mapsto (0,z_1+1/\epsilon, 1,\ldots
,1)$, which has length going to infinity when our perturbation
of the $N$-th point is going to zero.
\end{example}

From Lemma \ref{te1.1} we get Proposition \ref{standp}

\begin{proof}[Proof of Proposition \ref{standp}]
Proceed by induction over the number $N$ of points: For $N=1$ the
(parametrized translation) automorphism defined by $(w,z)\mapsto
(w,z-\zeta_1(w)+1)$ solves the problem in general, i.e. without
any assumption on the dimension $k$ of the parameter space.

Suppose the problem is solved for $N-1$ parametrized points. To
solve it for $N$ points take by induction assumption an automorphism $\alpha_1 \in
\mathrm{Aut}^k_{hol}(\C^n)$ which moves the first $N-1$ points to
their standard places
$$\alpha_1(w,\zeta_i(w))=(w,(i,0))\textrm{ for all } i=1,2,\ldots ,N-1 \textrm{ and for all } w\in \C^k.$$
It is not difficult to find an automorphism $\alpha_2\in
\mathrm{Aut}^k_{hol}(\C^n)$ which fixes the first $N-1$ standard
points, $\alpha_2(w,(i,0))=(w,(i,0))$ for all $i=1,2,\ldots ,N-1$
and for all $w\in \C^k$, such that the submanifold of $\C^{k+n}$
described by the (moved) last parametrized point
$$U:=\{ (w,\alpha_2 \circ \alpha_1\circ \zeta_N(w)):w\in \C^k\} \subset \C^{k+n}$$
is transversal to the (parametrized) $z_1$-axis
$$V:=\{ (w,(z_1,z_2,\ldots ,z_n))\in \C^{k+n}:z_2=z_3=\ldots =z_n=0\} .$$
Because of the dimension assumption this means that the two
submanifolds $U$ and $V$ do not meet ($\dim U +\dim
V=k+(k+1)<n+k=\dim \C^{k+n}$). In other words we are in the
position of Lemma \ref{te1.1} and find an automorphism
$\alpha_3\in  \mathrm{Aut}^k_{hol}(\C^n)$ fixing the first $N-1$
standard points and moving $(w,\alpha_2 \circ \alpha_1\circ
\zeta_N(w))$ to its standard place. The composition $\alpha_3
\circ \alpha_2 \circ \alpha_1$ is the desired automorphism moving
all $N$ parametrized points into their standard positions.
\end{proof}


We are now set to prove the interpolation lemma, Lemma \ref{FLYTTLEMMAT}.

\begin{proof}[proof of Lemma  \ref{FLYTTLEMMAT}]
Since by the dimension assumption $\dim X + k <n$ the points $b_i(w)$ are simultaneously standardizable(see Corollary \ref{standp}), we can
find $\alpha_1\in  \rm{Aut}^k_{hol}(\C^n)$ such that
$\alpha_1(w,b_i(w))=(w,(i,0,\ldots ,0))$ for every $w\in \C^k$.
Choose $R>0$ such that $\pi_2(\alpha_1(K))\subset R\B_n$ and 
choose $M$ such that $\alpha_1^{-1}(R\B_n )\subset M\B_n$.
 Let $C$ be a positive real constant such that the point  $\tilde p(w)=Cp(w)$ is outside of the ball of radius $M$
 for every $w\in \overline \Delta$. By transversality and our dimension condition \eqref{dim}
we can assume that an arbitrary small perturbation $\gamma_p$ of the path $p(w)+t(\tilde p(w)-p(w))$,
 $t\in [0,1]$, does not intersect
$\pi_2(\phi ( \overline \Delta \times \overline X_{R}))$ for any $w\in \overline \Delta$. 
Construct $\tilde q(w)$ and a path $\gamma_q$, in the same way
with the additional demand that the path for $q$ does not
intersect the path for $p$.

By the Andersén-Lempert-theorem with parameters and fixing
standard points (see Theorem \ref{alwpfc}) applied to the set $\alpha_1(K)$ union with a neighborhood
of the paths $\alpha_1(\gamma_p)$ and $\alpha_1 (\gamma_q)$ there
exists an $\alpha_2\in$  $ \rm{Aut}^k_{hol}(\C^n)$ such that
$\alpha_2(w,z)$ is close to the identity for $(w,z)\in K$   and
$|\pi_2(\alpha_2\circ \alpha_1(w,q(w)))|>R$ for every $w\in
\overline \Delta$. Furthermore we have $\alpha_2\circ
\alpha_1(w,b_i(w))=(w,(i,0,\ldots ,0))$ for every $w\in \C^k$.
Moreover $\alpha_2\circ \alpha_1(K)\subset R\B_n$. To be able to
apply Lemma \ref{te1.1} we perturb $\alpha_2\circ
\alpha_1(w,p(w))$ slightly to make sure that the last $n-1$
coordinate functions have no common zero on $\C^k$, at the same
time fixing the points $(w,(i,0,\ldots ,0))$, $i=1,\ldots ,N$.
This is possible by transversality and our dimension assumption
\eqref{dim}, i.e. $\dim X +k<n$. Now an application of Lemma \ref{te1.1} gives an
automorphism $\alpha_3 $ arbitrarily close to identity on
$\overline \Delta \times R \B_n$ fixing the first $N$ standard
points, such that $\alpha_3 \circ \alpha_2\circ
\alpha_1(w,p(w))=(w,2R,1,0,\ldots ,0)$ for every $w\in \C^k$.

By another application of the same lemma, we can in addition assume
that $\alpha_3 \circ \alpha_2\circ
\alpha_1(w,q(w))=(w,2R,2,0,\ldots ,0)$ for every $w\in \C^k$.

Finally we take a polynomial $\tilde P(z_1)$ such that $\tilde
P(2R)=1$,   $| \tilde P(z_1)|< \epsilon $ for $z_1\in
(R+1)\overline \B_1$ and $\tilde P(z_1)=0$ to order $l$, for
$z_1\in \{ 1,\ldots N\}$. Then define the automorphism
$$\alpha_4(w,z)=(w,z_1,z_2+ \tilde P(z_1),z_3, \ldots, z_n).$$
Consequently
$$\alpha_4 (w,i,0,\ldots ,0)=(w,i,0,\ldots ,0)$$
to prescribed order for $i=1,\ldots N$ and $\alpha_4 \circ
\alpha_3 \circ \alpha_2 \circ \alpha_1(w,p(w))=\alpha_3 \circ
\alpha_2 \circ \alpha_1(w,q(w))$.

In total,  the automorphism
$$\alpha =\alpha_1^{-1}\circ \alpha_2^{-1}\circ \alpha_3^{-1}\circ \alpha_4 \circ \alpha_3 \circ \alpha_2 \circ \alpha_1,$$  will have the properties stated in the lemma.

\end{proof}


\subsection{Proof of osculation lemma}

Remember that we will mark a finite number of points. The points will be used to exclude affine automorphisms in the main theorem.

Before we come to this point, a standard jet interpolation lemma in a parametrized form is established. This form is quite easy, since we do assume that the linear part of each prescribed 
jet is the identity. For general linear parts one would need (in order to get the linear part from shears)  to write a holomorphic (depending on the parameter in $ \C^k$) invertible matrix as a product of holomorphic elementary matrices. This is the so called Vaserstein problem posed by Gromov in \cite{Gro}. Although it was recently solved by Ivarsson and the first author  \cite{IK}, \cite{IK1} we will restrict ourselves to the present simple version of our lemma since it is fully sufficient for the purpose of the present paper.

\begin{lemma}\label{prekiss}
Let $\xi_i=(i,i,\ldots ,i)\in \C^n$ for $i=0,1,2,\ldots ,N$ and
let $P_i:\C^k\times \C^n\to \C^n$ be polynomial maps of degree $s$
such that
$$P_i(w,z)=\xi_i+(z-\xi_i )+O(|z-\xi_i|^2)$$
for $z\to \xi_i$ and all $w\in\C^k$.

Then there exists $\kappa \in \rm{Aut}_{hol}^k(\C^n)$ such that
$$\pi_2(\kappa (w,z)-(w,P_i(w,z)))=O(|z-\xi_i|^{s+1})$$
for $z\to \xi_i$ with $i=0,1,2,\ldots ,N$.
\end{lemma}

\begin{notation}
The reason that we have chosen $\xi_i=(i,i,\ldots ,i)\in \C^n$ for
$i=0,1,2,\ldots , N$ is to ensure that our points  have pairwise different projections along all coordinate directions.  This is suitable for the use of shears.\end{notation}

\begin{proof}

The proof goes exactly as in the non-parametric case (see Step 2.10 in Forstneric \cite{F}) by induction 
over the number of points and the order of the jets. The beginning step of the induction (first order) is empty in our case.
To realize  the homogeneous part $P^j$ of order $j \ge 2$  of a polynomial map $P$ by a composition of overshears on $\C^n$ depending holomorphically on the parameter on $w$ we need to establish the existence of finitely many linear functionals $\lambda_i$ together with vectors $v_i$ having the properties
$\lambda_i (v_i) = 0$ and $|v_i| = 1$ such that

\begin{equation}\label{vsum1}
\begin{split}
P^j (w,z) &=\sum_ic_i(w)(\lambda_i (z-\xi_1))^jv_i   + \\  &+\sum_k d_k(w)(\lambda_k (z-\xi_1))^{j-1}\langle
z-\xi_1,v_k\rangle v_k
\end{split}
\end{equation}

with holomorphic functions $c_i,d_i \in \hol (\C^k)$.

This follows from the  purely algebraic fact, Lemma \ref{basis} below.

\end{proof}

We denote the complex vector space $\C^n$ by $V$, $S^k (V^*)$ denotes the vector space of homogenous polynomials of
degree $k$ on $V$ and $r_{k,n} =  {n+k-1 \choose n-1}$ its dimension.

\begin{lemma}
There exist $r_{k,n}$ linear forms $\lambda_1, \lambda_2, \ldots, \lambda_{r_{k,n}} \in (\C^n)^*$ such that the homogenous
polynomials $(\lambda_i)^k$ of degree $k$ $i=1, 2, \ldots, r_{k,n}$ form a basis of $S^k (V^*)$. Moreover the $\lambda_i$
can be chosen from any nonempty open subset $\Omega$ of $(\C^n)^*$.
\label{l1}
\end{lemma}

\begin{proof} Take any nonzero element $\lambda_0 \in \Omega \in V^*$.
The map $V^* \to S^k (V^*)$ defined by $\lambda \mapsto \lambda^k$
is $\Gl (V)$-equivariant and since $S^k (V^*)$ is an irreducible
$\Gl (V)$-module the linear span of the $\Gl (V)$-orbit through
$\lambda_0$ $$\spa \{ (g\cdot \lambda_0)^k, \quad g\in \Gl (V)\}$$
is the whole module $S^k (V^*)$. The same holds for any open part
of the orbit, i.e.
$$\spa \{ (g\cdot \lambda_0)^k,\quad g\in U\} = S^k (V^*)$$
for any open subset $U$ of $\Gl (V)$, since if the left hand side would be contained in some proper linear subspace
$W \subset S^k (V^*)$ then by the identity theorem for holomorphic mappings the whole orbit would be contained in $W$
contradicting the irreducibility of  $S^k (V^*)$. We can therefore find $r_{k,n}$ group elements $g_1, g_2, \ldots, g_{r_{k,n}}$
$\in \Gl (V)$ contained in some open neighborhood $U$ of the identity (with $U\cdot \lambda_0 \subset \Omega$) such that
the homogenous polynomials $(g_j \cdot \lambda_0)^k$ $k=1, 2, \ldots, r_{k,n}$ form  a basis of $S^k (V^*)$.
\end{proof}

\begin{lemma}
There exist $n \cdot {n+k-2 \choose n-1}- {n+k-2 \choose n-1}$ linear forms $\lambda_i \in (\C^n)^*$ and vectors $v_i\in \C^n$ with $\lambda_i (v_i)=0$ and
$\Vert v_i\Vert =1 $ $i=1, 2, \ldots,  {n+k-2 \choose n-1}$ together with ${n+k-2 \choose n-1}$
linear forms $\tilde\lambda_j \in (\C^n)^*$ and  vectors $w_j \in \C^n$ with
$\lambda_j (w_j)=0$ and
$\Vert w_j\Vert =1 $ $j=1, 2, \ldots,  n \cdot {n+k-2 \choose n-1}- {n+k-2 \choose n-1}$ such that the homogenous
polynomial maps

$$z\mapsto (\lambda_i (z))^k v_i, \quad  i=1, 2, \ldots, n \cdot {n+k-1 \choose n-1}- {n+k-2 \choose n-1}$$

of degree $k$ together with the homogenous polynomial maps

$$ z\mapsto (\tilde\lambda_j (z))^{k-1} \langle z, w_j\rangle w_j, \quad j=1, 2, \ldots, {n+k-2 \choose n-1}$$

of degree $k$ form a basis of the vector space $V_k \cong  S^k ((\C^n)^*) \otimes \C^n$ of homogenous polynomial maps
of degree $k$. Moreover if $v_0 \in \C^n$ and a non-zero functional $\lambda_0 \in (\C^n)^*$ with $\lambda_0 (v_0)=0$ and
$\Vert v_0\Vert =1 $ and a number $\epsilon >0$ are given, then the vectors $v_i, w_j$ together with the functionals $\lambda_i,\tilde \lambda_j$ can be chosen with
$\Vert v_0 -v_i\Vert  <\epsilon$ $\Vert v_0 -w_j\Vert  <\epsilon$ and $\Vert \lambda_0 - \lambda_i \Vert <\epsilon$, $\Vert \lambda_0 -\tilde \lambda_j \Vert <\epsilon$. 
\label{basis}
\end{lemma}

\begin{proof}
Set $V=\C^n$. The vector space $V_k \cong S^k (V^* ) \otimes V$ is as $\Gl (V)$-module isomorphic to the direct sum
of two irreducible representations $W_1 \oplus W_2$, where $W_1$ is isomorphic to $S^{k-1} (V^* )$ and $W_2$ is
isomorphic to the kernel of the $\Gl (V)$-equivariant map
$$\psi : S^k (V^*) \otimes V \to S^{k-1} (V^*), \quad \quad X\mapsto \div X .$$

We will provide some ``section'' of $\psi$:
\smallskip\noindent
By Lemma \ref{l1} there exist $\tilde\lambda_j \in V^*$ $j=1, 2, \ldots, {n+k-2 \choose n-1}$ ($\epsilon$-near to $\lambda_0$ if desired) such that the
homogenous polynomials $\tilde\lambda_i^{k-1}$ form a basis of $S^{k-1} (V^*)$. Choose vectors $w_j$ with
$\tilde\lambda_j (w_j) = 0$ and $\Vert w_j \Vert =1$ (and $\epsilon$-near to $v_0$ if desired). For a homogenous polynomial $p(z)$ of degree $k-1$ we write
it in the basis $p(z) = \sum_{j=1}^{n+k-2 \choose n-1} d_j (\tilde\lambda_j(z))^{k-1}$ and define the section
$s(p)$ by

$$s(p) (z) = \sum_{j=1}^ {n+k-2 \choose n-1} d_j (\tilde\lambda_j(z))^{k-1} \langle z, w_j\rangle w_j.$$

An easy calculation shows $\psi (s(p))) = \div s(p) = p$. Thus the homogenous polynomial maps
$ z \mapsto (\tilde\lambda_j(z))^{k-1} \langle z, w_j\rangle w_j$ $j= 1, 2, \ldots,{n+k-2 \choose n-1}$ form
a basis of some linear subspace of $S^k (V^*) \otimes V$ complementary to the kernel of $\psi$.
\smallskip
Now take our nonzero linear functional $\lambda_0 \in V^*$ and some vector $v_0 \in V$ with
$\lambda_0 (v_0) = 0$ and $\Vert v_0 \Vert = 1$.
Since $\ker \psi \cong W_2$ is an irreducible $\Gl (V)$-module the linear span of any $\Gl (V)$-orbit
through a nonzero point in $W_2$ is the whole vector space $W_2 \cong \ker \psi$. Since $\lambda_0 (v_0) =0$
the element $\lambda_0^k \otimes v_0$ is such a point and like in the proof of Lemma \ref{l1} we find
group elements $g_1, g_2, \ldots , g_{  n \cdot {n+k-1 \choose n-1}- {n+k-2 \choose n-1}}\in \Gl (V)$ contained
in any given nonempty open neighborhood of the identity element such that the homogenous polynomial maps
$g_j \cdot (\lambda_0^k \otimes v_0) = (g_j \cdot \lambda_0) g_j\cdot v_0$
form a basis of $\ker \psi$. Defining $\\lambda_i = g_i \cdot \lambda_0$ (remember $g \cdot \lambda_0 (v) :=
\lambda_0 ( g^{-1} v)$) and $\tilde v_i = g_i \cdot v_0, v_i= \frac { \tilde v_i}{\vert \tilde v_i \vert}$ (instead of normalizing we could have chosen the $g_j$ from the unitary group since by the identity principle in complex analysis any irreducible $\Gl_n$-representation is $\rm{U}_n$-invariant) we get a basis $\lambda_i^k \otimes v_i$ $i=1, 2, \ldots,
n \cdot {n+k-1 \choose n-1}- {n+k-2 \choose n-1}$ of $\ker \psi$. Together with the above constructed
basis of the complementary subspace it forms a basis of the vector space $V_k \cong  S^k ((\C^n)^*) \otimes \C^n$ of homogenous polynomial maps
of degree $k$.
\end{proof}

Using these prerequisits we give the proof of Lemma \ref{KYSSLEMMAT}

\begin{proof}[proof of Lemma  \ref{KYSSLEMMAT}]
Since the points $\xi_1(w),\ldots ,\xi_t(w)$ are simultaneously
standardizable we can assume that $\xi_i(t)=(i,i,\ldots ,i)$ for
$1\leq i \leq t$. Now we want to apply Lemma \ref{prekiss} to make
$\kappa (M(w))$ osculating of order $l$. This means we have to
ensure that there is for each $i$ a holomorphically depending on
$w\in \C^k$ polynomial map $P_i(w,z)$ with
\begin{equation}\label{jetform}
P_i(w,z)=\xi_i+(z-\xi_i )+O(|z-\xi_i|^2).
\end{equation}
We then get osculation of order $l$ at $\xi_i(w)$ after applying
an automorphism with this prescribed jet at $\xi_i(w)$.

For a given point $x$ in a submanifold $M$ of $\C^n$ the set
$P(x,M)$ of $l$-jets of the form \eqref{jetform} ensuring the
osculation up to order $l$ is biholomorphic to a vectorspace and
the change of variables in the jet-bundle is affine-linear. This
means we have to find a holomorphic section in a locally trivial
fibration over $\C^k$ of the form
$$\begin{matrix}
\bigcup_{w\in \C^k}P(\xi_i(w),M(w))\\
\downarrow \\
\C^k
\end{matrix}$$
where the fibers are biholomorphic to a vectorspace $\C^N$ and the
structure group is $\mathrm{Aff}(\C^N)$, the group of affine
linear automorphisms of $\C^N$. Since $\C^N$ is contractible a
continuous section always exists and the Oka-Grauert principle (Theorem \ref{oka})
implies the existence of a holomorphic section.  
\end{proof}


\subsection{The proof of Lemma \ref{AKlemma}}

To prove Lemma \ref{AKlemma} we use the following sublemma:
\begin{lemma}\label{sublemma}
For every point $p=(w_0,\hat p)\in K_M=K_1\times K_2$ there is an
open neighborhood $V_p\times U_p\ni p$ in $K_1\times M$ and a
family of automorphisms $\psi_t$ of $\C^k\times \C^n$ parametrized
by $\C^{N}$, where $N=N(m,n,s)=\paren{ \binom{m+s}{m}-(m+1)}(n-m)$, such that
\begin{enumerate}[1.]
\item $\psi_0=\rm{Id}$. \item Every $\psi_t$ satisfies 1. and 2.
in Lemma \ref{AKlemma}. \item There exist an open neighborhood $T$
of $0$ in $\C^N$ such that
\begin{multline*}
\Sigma =\{ t\in T: \textrm{ There exists } p'\in V_p\times U_p \textrm{ such that }\psi_t(w,M(w))\\
\textrm{ osculates of order $l$ in } \psi_t(p') \} 
\end{multline*}
is a set of Lebesgue measure zero.
\end{enumerate}
\end{lemma}

\begin{proof}
If $M(w)$ does not osculate of order $l$ in $p$ let
$\psi_t=\rm{Id}$ for every $t\in \C^N$.

Now suppose that $M(w)$ osculates of order $l$ in $p$.  Without
loss of generality  assume that $p=(w_0,0)$ and that the tangent
plane is given by $T_{(w_0,0)}M(w)=\{ (z_1,\ldots ,z_m,0,\ldots
,0)\}$. Let $\pi_1:\C^n \to  \C^m$ denote the projection to the
$m$ first coordinates of $\C^n$. After a linear change of
variables we can assume that $\pi_{1}(b_i(w))\neq 0\in  \C^m$ for
$1\leq i\leq q$ and for $w\in \tilde V_p$ where $\tilde V_p$ is
some open neighborhood in $K_1$.

To control if $\psi (M(w))$ osculates of order $l$ in some point
$\psi (p')$, $p'\in \C^k \times M(w)$ for a given automorphism
$\psi \in \rm{Aut}^k(\C^n)$, consider the map $F^{\psi}:\C^k\times
M(w)\to \C^N$ where the coordinate functions $F^{\psi}_{\alpha
,u}$, enumerated by pairs $(\alpha ,u)$, where $\alpha$ is a multi
index $\alpha =(\alpha_1,\alpha_2,\ldots ,\alpha_m)$ with $2\leq
|\alpha |\leq l$ and $u\in \N$ satisfies that $m+1\leq u \leq n$,
are given by
\begin{equation*}
F^{\psi}_{\alpha ,u}(w,\zeta )=\det
\begin{pmatrix}
\frac{\partial}{\partial \zeta_1}(\psi )_1(w,\zeta ) & \dots & \frac{\partial}{\partial \zeta_1}(\psi )_m(w,\zeta ) &\frac{\partial}{\partial \zeta_1}(\psi )_u(w,\zeta ) \\
\vdots & \ddots & \vdots & \vdots \\
\frac{\partial}{\partial \zeta_m}(\psi )_1(w,\zeta ) & \dots & \frac{\partial}{\partial \zeta_m}(\psi )_m(w,\zeta ) &\frac{\partial}{\partial \zeta_m}(\psi )_u(w,\zeta ) \\
\frac{\partial}{\partial \zeta^{\alpha}}(\psi )_1(w,\zeta ) &
\dots & \frac{\partial}{\partial \zeta^{\alpha}}(\psi )_m(w,\zeta
) &\frac{\partial}{\partial \zeta^{\alpha}}(\psi )_u(w,\zeta )
\end{pmatrix}
\end{equation*}
Here $(\psi )_i$ denotes the $i$-th $z$-coordinate function of the
map $\psi :\C^k\times \C^n \to \C^k\times \C^n$ and
$(\zeta_i)_{1\leq i\leq m}$ is some fixed system of local
coordinates of $M(w)$ near $p$. Then $\psi (M(w))$ osculates of
order $l$ in some point $(w,\zeta )$ if and only if
$F^{\psi}(w,\zeta )=0$.

If we restrict our attention to a small enough neighborhood
$\tilde U_p\in K_2$, containing $0$, we can use $z_1,\ldots
,z_m$ as local coordinates on $M(w_0)$ near $p$, in fact we will use
the coordinates $(w,z_1,\ldots ,z_m)$ in the restriction to
$\tilde V_p\times \tilde U_p$. To construct the family of
automorphisms $\psi_t$ of $\C^k\times \C^n$ we do the following:
For every pair $(\alpha ,u)$ we choose a holomorphic function
$h_{\alpha , u}$ on $\C^k\times \C^m$ such that
\begin{enumerate}[1.]
\item $(w,h_{\alpha ,u}(w,z))=(w,z^{\alpha})$ of order at least
$l+1$ in $(w_0,0)$. \item $(w,h_{\alpha
,u}(w,\pi_1(b_i(w))))=(w,0)$ of order at least $l+1$ for  $1\leq
i\leq q$ and for all $w\in \C^k$. \item $(w,h_{\alpha
,u}(w,\pi_1(a_i))=(w,0)$ for  $1\leq i\leq r$ and for all $w\in
\C^k$.
\end{enumerate}
Now define the map $\psi :\C^N\times \C^k \times \C^n \to \C^k
\times \C^n$ by
$$\psi(t,w,z)=(w,z+\sum_{(\alpha ,u)}t_{(\alpha ,u)}h_{\alpha ,u}(w,z_1 ,\ldots ,z_m)e_u),$$
where $e_u$ is the $u$-th unit vector, $m+1\leq u \leq n$.

This construction gives that for every $t\in \C^N$ the map
$\psi_t=\psi (t,\cdot )$ is a parametrized automorphism of $\C^n$
and because of condition 2. and 3. $\psi$ fulfills conditions 1.
and 2. of Lemma \ref{AKlemma}. Furthermore $\psi_0=\rm Id$. The
only thing left to check is that $\psi$ fulfills condition 3..

Using the fact that $h_{\alpha ,u}(w,z)=z^{\alpha}$ we get that
\begin{multline*}
\frac{\partial}{\partial t_{(\alpha ,u)}}F^{\psi_t}_{\alpha ,u}|_{\substack{z=0\\ w_0=0}}=-\frac{\partial}{\partial t_{(\alpha ,u)}}\sum t_{(\alpha ,u)}\frac{\partial}{\partial z^{\alpha}}h_{\alpha ,u}(w,z)|_{\substack{z=0\\ w_0=0}}=\\
=\frac{\partial}{\partial t_{(\alpha ,u)}}\sum t_{(\alpha
,u)}\frac{\partial z^{\alpha}}{\partial
z^{\alpha}}|_{\substack{z=0\\ w_0=0}}=\alpha !.
\end{multline*}
Moreover the derivative with respect to $t$ depending on other
pairs $(\alpha ',u')$ will vanish
\begin{multline*}\frac{\partial}{\partial t_{(\alpha ' ,u')}}F^{\psi_t}_{\alpha ,u}|_{\substack{z=0\\ w_0=0}}=-\frac{\partial}{\partial t_{(\alpha ',u')}}\sum t_{(\alpha ,u)}\frac{\partial}{\partial z^{\alpha}}h_{\alpha ,u}(w,z)|_{\substack{z=0\\ w_0=0}}=\\
=\frac{\partial}{\partial t_{(\alpha ',u')}}\sum t_{(\alpha
,u)}\frac{\partial z^{\alpha}}{\partial
z^{\alpha}}|_{\substack{z=0\\ w_0=0}}=0,
\end{multline*}
whenever $u\neq u'$ or whenever $u=u', |\alpha '|\leq |\alpha |$
and $\alpha '\neq \alpha$.

This implies that the map $\Phi :\C^N\times \C^k\times M\to \C^N$
defined by $\Phi (t,w,z)=F^{\Phi_t}(w,z)$ has maximal rank near
$(0,w_0,0)=(0,p)$. Thus there exists an open neighborhood
$\Omega_p$ of the form $\Omega_p=T\times V_p\times U_p$ of $(0,p)$
in $\C^N\times \C^k\times M$ such that $\Phi |_{\Omega_p}$ is
transversal to $0\in \C^N$. This implies that for almost all $t\in
T$ the map $F^{\psi_t}:\C^k\times M(w)\to \C^N$ is transversal to
$0$. Since $m<N$ this means that for almost all $t\in T$ the image
$F^{\psi_t}(V_p\times U_p)$ does not meet $0$, i.e.
$\psi_t((w,M))$ does not osculate of order $l$ for any $p'\in
V_p\times U_p$.
\end{proof}

\begin{proof}[Proof of Lemma \ref{AKlemma}]
Choose finitely many open subsets $V_i\times U_i$ of $\C^k\times
M$ together with families $\psiî:T_i \times \C^k\times \C^n \to
\C^k\times \C^n$ of automorphisms $i=1,2,\ldots ,r$ as in Lemma
\ref{sublemma} and choose compact subsets $K_i\subset  V_i\times
U_i$ of $V_i\times U_i$ which cover $K_M$. Since $\psi^1_0$ is the
identity, for $t$ sufficiently small the automorphism $\psi^1_t$
moves no point of $K_M$ more than $\frac{\epsilon}{r}$. So we find
$t_1\in T_1$ such that $|P_2
\phi^1_{t_1}(w,z)-z|<\frac{\epsilon}{r}$ for every $(w,z)\in K$
and the submanifold $\psi^1_{t_1}(\C^k\times M)$ does not osculate
of order $l$ at any point of $\psi^1_{t_1}(K_1)$.

Observe that the property of not osculating of order $l$ at some
point is preserved under small perturbations, i.e. for each
compact subset $L$ of a submanifold $M$ of $\C^n$ which does not
osculate of order $l$ at any point of $L$ there exists some
$\epsilon >0$ such that for each automorphism $\psi$ of $\C^n$ the
property $|\psi (z)-z|<\epsilon$ for every $z\in L$ implies that
$\psi (M)$ remains non-osculating of order $l$ at any point of
$\psi (L)$ (for holomorphic maps small perturbations in values
imply small perturbations in derivatives).

Hence we find a sufficiently small $t_2\in T_2$ such that first
$|P_2 \psi^2_{t_2}(w,z)-z|<\frac{\epsilon}{r}$ for every $z\in
\psi^1(K)$, second the submanifold $\psi^2_{t_2}\circ
\psi^1_{t_1}(\C^k\times M)$ does not osculate of order $l$ at any
point of $ \psi^2_{t_2}\circ \psi^1_{t_1}(K_2)$ and third
$\psi^2_{t_2}\circ \psi^1_{t_1}(\C^k\times M)$ remains
non-osculating of order $l$ at any point of $ \psi^2_{t_2}\circ
\psi^1_{t_1}(K_1)$. Proceeding by induction we find an
automorphism $\psi :=\psi^r_{t_r}\circ \psi^{r-1}_{t_{r-1}}\circ
\cdots \circ \psi^1_{t_1}$ moving no point of $K$ more than
$\epsilon$ and such that $\psi (\C^k\times M)$ does not osculate
of order $l$ at any point of $\psi \left( \cup_{i=1}^rK_i\right)
\supset \psi (K_M)$. Since all automorphisms $\psiî_t$ satisfy
properties 1. and 2., $\psi$ satisfies them as well.
\end{proof}

\end{document}